\documentclass[11pt, a4paper, leqno]{amsart}
\usepackage{esint}
\usepackage{color}

\usepackage{amsthm}
\usepackage{amsmath}
\usepackage{amsfonts}
\usepackage{amssymb}
\usepackage{amscd}
\usepackage{mathrsfs}
\usepackage{enumerate}
\usepackage{enumitem}

\usepackage[utf8]{inputenc}
\usepackage[english]{babel}
\usepackage{anysize}
\usepackage{wasysym} %zodiaco
\usepackage{mathrsfs} %mathscr
\usepackage{bbm} %mathbbm
\usepackage{calligra} %calligra
\usepackage{enumerate} %(i)

\newcommand{\R}{\mathbb{R}}

\newcommand{\supp}{\operatorname{supp}}
\newcommand{\map}{\mathbf{x}}

\newcommand{\InterphaseMix}{\Omega_{\mathrm{mix}}}

\newcommand{\function}[5]{\begin{eqnarray*}
		#1:#2 & \rightarrow & #3 \\
		   #4 & \mapsto     & #5
\end{eqnarray*}}

\newcommand{\Hil}[1]{H^{#1}}

\newtheorem{prop}{Proposition}
\newtheorem{lemma}{Lemma}

\newtheorem{definition}{Definition}
\newtheorem{corol}{Corollary}
\newtheorem{teor}{Theorem}
\theoremstyle{remark}

\theoremstyle{remark}
\newtheorem{remark}{Remark}

\author[Víctor Arnaiz]{Víctor Arnaiz}

\address{Université Paris-Saclay, CNRS, Laboratoire de mathématiques d’Orsay, 91405, Orsay, France.}

\email{victor.arnaiz@universite-paris-saclay.fr}

\author[Ángel Castro]{Ángel Castro}

\address{Instituto de Ciencias Matemáticas, CSIC-UAM-UC3M-UCM, E-28049, Madrid, Spain.}

\email{angel\_castro@icmat.es}

\author[Daniel Faraco]{Daniel Faraco}

\address{Departamento de Matemáticas, Universidad Autónoma de Madrid, E-28049, Madrid, Spain.}

\email{daniel.faraco@uam.es}

\title[Semiclassical estimates for pseudodifferential operators]{Semiclassical estimates for pseudodifferential operators and the Muskat problem in the unstable regime}

\begin{document}

\begin{abstract}
We obtain  new semiclassical estimates for pseudodifferential operators with low regular symbols. Such symbols appear naturally in a Cauchy Problem related to recent weak solutions to the unstable Muskat problem constructed via convex integration in \cite{CCF16}.  In particular, our new estimates reveal the tight relation between
the speed of opening of the mixing zone and the regularity of the interphase.
\end{abstract}

\maketitle

\section{Introduction and main results}

\noindent The evolution of a two fluids through a porous media where one of the fluid is above the other  is known as the Muskat problem \cite{Muskat37}. The physical derivation builds on the conservation of mass, the incompressibility of the flow and the Darcy law, which relates the velocity with the forces, namely the pressure and gravity. Such system is known
as the IPM system. Let us further assume that the fluids have constant densities and equal viscosities,  the permeability of the medium is also constant, and the initial data is given by the graph of a function $f_0:\mathbb{R} \to \mathbb{R}$.
That is \begin{equation}\label{initialdensity}
\rho|_{t=0}= \rho_1\chi_{\Omega_M}+\rho_2(1-\chi_{\Omega_M}),
\end{equation}
where $\Omega_{M}$ is the epigraph of $f_0$. If we  make the ansatz that as time
evolve the fluid still consists of two fluids separated by an smooth interphase $f(t,x)$, then $g=\partial^4 f$ has to solve a nonlinear and non local equation
\begin{equation}
\label{Muskat} \partial_t g=p_{M}(x,D) g+T[f](\partial_x g)+ R[f](  x,t).
\end{equation}
Here,
\begin{equation}
\label{e:muskat_symbol}
p_{M} (x,\xi)= -(\rho_2-\rho_1) \frac{|\xi|}{1+|f'(x)|^2},
\end{equation}
and $p_M(x,D)$ stands for the canonical quantization of $p_M$ (see \eqref{e:canonical_quantization} below),  $T[f]$ can be thought as a smooth function, and $R$ is a lower order remainder.

It turns out that if $\rho_1<\rho_2$ then the system is well posed for sufficiently regular $f$ (see \cite{Contour,CGS16,CGSV17,Mat19}). However, if $\rho_2>\rho_1$, the system is ill posed \cite{CCFL12,Contour}. In this situation, starting with the pioneering  work of Saffman and Taylor \cite{ST58}, numerics and experiments \cite{AT83,MH95} predict a fingering pattern in the evolution and the existence of an evolving in time \textbf{mixing zone} $\InterphaseMix$, where the fluids mix chaotically and the pointwise (microscopical) pattern of the two fluids is practically unpredictable.  However, as pointed out by Otto (see among others \cite{Otto97,Otto99,Otto2001}) several aspects of the  mixing zone and of the mixing pattern of the fluids can be derived from the relaxation of the system.

Recently, the IPM system and the Muskat problem have been revisited using
DeLellis-Székelyhidi program  to apply convex integration in hidrodynamics \cite{CFG11, Sze12,CCF16, FSz18, CFM19} (see
\cite{DSz12} for a review of the method). In particular, the various constructions of weak solutions to the Muskat problem yield an explicit description of the mixing zone.

In  \cite{CCF16} the mixing zone is described as a neighborhood of width $t c(t,x)$ of a pseudo-interphase $f(t,x)$ evolving in time $t \in [0,T]$.  More precisely, the map
\function{\map}{\R\times[-1,1]\times[0,T]}{\R^2}{(x,\lambda,t)}
{\big( x,f(t,x)+ \lambda t c(t,x) \big)}
defines the mixing zone $\InterphaseMix$ as
\begin{equation}\label{IMix}
\InterphaseMix(t)=\map  \big(\R\times(-1,1),t \big),\quad t\in(0,T].
\end{equation}
Moreover, in \cite{CCF16,CFM19} it has been proven that if $f$ and $c$ are suitable coupled through the equation
\begin{align}\label{Muskatmix}
\partial_t f = \mathcal{M}[c,f] f,
\end{align}
where $\mathcal{M}[c,f] f$ is a nonlinear integro-differential operator acting on $f$,
 then there exist infinitely many weak solutions to the IPM system compatible with such mixing zone (called mixing solutions).  For $g=\partial^4 f$ it can be checked that
\begin{equation}
\label{full}
\partial_t g=p_{\textrm{mix}}(x,D)(g) + T[f]\partial_x g +R[f],
 \end{equation}
where, analogously to \eqref{Muskat}, $T[f]$ is a smooth function and $R$ are lower order terms.
The symbol $p_{\textrm{mix}}$ has a rather cumbersome explicit  expression, see \cite{CCF16}, but  it satisfies that $p_{\textrm{mix}}(x,\xi) \approx  d_t(x,\xi)$, where the symbol $d_t(x,\xi) = t^{-1}d(x,t \xi)$  and

\begin{equation}
\label{e:differential_symbol}
d(x,\xi) := \frac{\vert \xi \vert}{1 + c(x)\vert \xi \vert}.
\end{equation}
Here $c(x)$  is a smooth function which is assumed to not depend on $t$ for simplicity. It turns that $p_{\textrm{mix}}$ in \eqref{full} is  slightly better than $p_{M}$ in \eqref{Muskat} and this is the reason why \eqref{Muskatmix} can be solved as is proved in \cite{CCF16}.

This paper is focused on the  study of the equation

\begin{equation}\label{IVP}
\text{(IVP) }
\left \lbrace \begin{array}{ll}
\partial_t f(t,x) = d_t(x,D) \, f(t,x), \\[0.2cm]
  f(0,\cdot) = f_0 \in L^2(\R).
\end{array} \right.
\end{equation}
which captures the main difficulties in order to get a new energy estimate for \eqref{full} which allows us to show local existence for \eqref{Muskatmix} with an improvement of the regularity with respect to \cite{CCF16}.

Notice that if $c$ is identically constant, then (IVP) has a global-in-time solution which can be computed explicitely using the Fourier transform. Indeed,
\begin{equation}
\label{e:solution_c_constant}
f(t,x) = p(tD)f_0(x) = \int_\R (1 + c \,  t \vert \xi \vert)^{\frac{1}{c}} \widehat{f}_0(\xi) e^{2\pi i x \xi} d\xi, \quad t> 0.
\end{equation}

In particular, the following conservation law holds:
\begin{equation}
\label{e:conservation_with_constant_c}
\frac{d}{dt} \Vert p^{-1} (tD) f(t) \Vert_{L^2}^2 = 0, \quad \forall t \geq0,
\end{equation}
and $f(1,x)$ is comparable to the $c^{-1}$ derivative of the initial data. Thus the regularity of the solution seems to decreases as the width of the mixing zone is thinner. In \cite{CCF16} a G\aa rding-inequality is used to deal with a variable with $c(x)$ and this yields  a loss of one derivative with respect to the initial data. Here we frontally attack (IVP) via a suitable new commutator estimates.  This new strategy gives us  a gain of local regularity.

Notice also that the operator $d_t(x,D)$, written in (IVP) as the canonical quantization of $d_t$, can be also put in the form $d_t(x,D) = t^{-1} d(x,tD)$, where now $d(x,tD)$ denotes the semiclassical quantization of $d$ with semiclassical paramter $t$ (see \eqref{e:semiclassical_quantization} below). This precision may seem at first view merely cosmetic. However, one of the goals of this work is to clarify that semiclassical calculus, and not only pseudodifferential calculus, is essential to obtain analogous conservation laws to \eqref{e:conservation_with_constant_c} for the solution of (IVP) in the non-constant case.

The study of such evolutions with low regular and variable growth symbols seems to be new in the semiclassical picture and might find applications elsewhere. Interest in low regular symbols appear in other problems in fluid mechanics \cite{Lannes06,Tex07}. In \cite{Lannes06}, Lannes studies in a very careful  way the action of pseudodifferential operators $a(x,D)$ on Sobolev spaces $H^s$, with symbols $a(x,\xi)$ having limited regularity in the position variable $x$ and in momentum variable $\xi$ near the origin,  via the use of paradifferential calculus. In \cite{Tex07}, Texier extends some of the techniques of \cite{Lannes06} to deal with semiclassical pseudodifferential operators having only low regularity in the $x$ variable and being smooth in the $\xi$ variable. Our result is indeed related to these, but we need to use semiclassical calculus with symbols having very low regularity in $x$ and in $\xi$. This, up to our view, entails certain obstructions to extend the techniques of \cite{Lannes06} to the semiclassical framework through the use of paradifferential calculus, and for this reason we only use it tangentially. Alternatively, our approach is strongly concerned with the techniques of \cite{Hwang87, CCF16}.

We present two semiclassical theorems in the form of conservation laws that predict the $c(x)^{-1}$ loss of regularity with respect to the local regularity of the initial data, in contrast with the loss of one derivative obtained in \cite{CCF16}. To this aim, we define the symbol
\begin{equation}
\label{e:p}
p(x,\xi) := (1 + c(x) \vert \xi \vert)^{\frac{1}{c(x)}}.
\end{equation}
In order to state our results, we need to impose some regularity assumptions on $c(x)$. The precise class of \textit{admissible} functions $c(x)$ considered in this work is fixed in Definition \ref{d:admissible} below.  We first state a local-in-time conservation law in terms of the pseudo-inverse of $p(x,tD)$:
\begin{teor}
\label{t:main_theorem1}
Let $c_1,c_2 > 0$ be an admissible pair satisfying \eqref{e:admissible_pair}. Let $c(x)$ be $(c_1,c_2)$-admissible. Then there exists $T > 0$ such that
$$
\frac{d}{dt} \Vert p^{-1}(x,tD) f(t) \Vert^2_{L^2(\R)} \leq C_T \Vert p^{-1}(x,tD) f(t) \Vert_{L^2(\R)}^2, \quad \forall t \in (0,T],
$$
where the constant $C_T$ depends only on $T$ and on $c$. In particular, $\Vert p^{-1}(x,D) f(t) \Vert_{L^2(\R)}$ remains bounded for $t \in [0,T]$ if $f_0 \in L^2(\R)$.
\end{teor}

\begin{remark}
%An important admissible pair is $(c_1,c_2)=(1,2)$, which is the one considered in  \cite{CCF16}. In contrast, if $(c_1,c_2)$ is an admissible pair with $c_1$ and $c_2$ close to zero, then Theorem \ref{t:main_theorem1} illustrate how the regularity of $c(x)$ is linked with the width of the mixing zone. Precisely, as the mixing zone is thinner, the regularity of $c(x)$ is required to increase, while the loss of derivatives of $f(t,x)$ becomes larger. This means that the regularity of $f(t,x)$ in the $x$ variable is related to $c(x)^{-1}$ via the pseudo-inverse $p^{-1}(x,tD)$, and hence it appears as a pseudo-local feature.
Theorem \ref{t:main_theorem1} illustrates how the size of $c(x)$ is linked with the  regularity of the pseudo-interface $f(x,t)$ . Precisely, as the coefficient $c_1$ is smaller, the regularity of $c(x)$ is required to increase, while the loss of derivatives of $f(t,x)$ with respect to $f(0,x)$ becomes larger. This means that the regularity of $f(t,x)$ in the $x$ variable is related to $c(x)^{-1}$ via the pseudo-inverse $p^{-1}(x,tD)$, and hence it appears as a pseudo-local feature.
\end{remark}

Our next result explains the local smoothing properties of $p^{-1}(x,tD)$ around any fixed point $x_0 \in \R$ in terms of local Sobolev regularity. We take $\varepsilon > 0$ small and define
\begin{align*}
c_\varepsilon^+  := \sup_{x \in I_\varepsilon} c(x) + \varepsilon; \quad c_\varepsilon^- & :=  \inf_{x \in I_\varepsilon} c(x) - \varepsilon,
\end{align*}
where $I_\varepsilon = (x_0 - \varepsilon, x_0 + \varepsilon)$.
\begin{teor}
\label{t:localization}
Let $c_1,c_2 > 0$ satisfying \eqref{e:admissible_pair}. Let $c(x)$ be $(c_1,c_2)$-admissible. Set $m_1 = c_1^{-1} $,  $s_- := -1/c_\varepsilon^-$ and $s_+ := -1/c_\varepsilon^+$. Then, for every smooth bump function $\chi_\varepsilon(x)$ supported on $I_\varepsilon$, there exist constants $C_0, C_1,C_2 > 0$ such that
$$
C_1 \Vert \chi_\varepsilon f \Vert_{H^{s_-}_t(\R)} - t C_0 \Vert f \Vert_{H^{-m_1}_t(\R)} \leq  \big \Vert \chi_\varepsilon(x) p^{-1}(x,tD) f \big \Vert_{L^2}  \leq C_2 \Vert \chi_\varepsilon f \Vert_{H^{s_+}_t(\R)} + t C_0 \Vert f \Vert_{H^{-m_1}_t(\R)},
$$
for every $f \in H^{-m_1}_t(\R)$ and $0 < t \leq T$.
\end{teor}

In principle these theorems combined with the strategy from \cite{CCF16} should yield the corres\-ponding $c(x)^{-1}$ loss of
regularity  for the more complicated equation \eqref{full} as well as the price of reproducing some of the heavy computations of
\cite{CCF16} (we remark that this is consistent with the fact that the unstable Muskat problem $c \equiv 0$ is ill posed).
  A solution to
\eqref{full} combined  with \cite[Theorem 4.2]{CCF16} yields the existence of a continuous subsolution. Then the h-principle \cite{CFM19} yields  weak solutions whose averages are essentially the subsolution. Indeed, these averages, which represent how the fluids mix in a mesoscopic scale,  are continuous and strictly monotone from the top to the bottom of the mixing zone which seems a natural feature. Let us emphasize that, in this approach, the regularity of the boundary of mixing zone at a point $(x, f(x))$ just depends  on the speed of opening $c(x)$ in
a neighborhood of $x$.

Let us remark that the construction of the mixing zone (and of the corresponding subsolutions and solutions)  is highly non unique  and in these various problems selecting a one which prevales above the others based on physical principles (diffusion, surface tension, entropy rate maximizing) is perhaps the most challenging problem \cite{OM06,YS06,Sze12}.
Remarkably, in   the approach of \cite{FSz18,NSz20},  which provides piecewise constant subsolutions, the speed of opening $c(x)$ does not affect the regularity of the mixing zone, and it is only present in the norm of the operators.

%if the initial data is such that

%where $f\in\Cont{}([0,T];\Hil{4}(\R))$ is a suitable evolution of $f_0\in H^5(\R)$ (see \cite[(1.11)]{Mixing}) and $c>0$ is the speed of growth of the mixing zone. At each $t\in(0,T]$, the mixing zone $\InterphaseMix(t)$
%splits $\R^2$ into two open connected sets $\Omega_{\pm}(t)$ defined from
%$\partial\Omega_{\pm}(t)=\map(\R,\pm 1,t)=\mathrm{Graph}(f(t))\pm (0,ct)$.

%\begin{align}
%\partial_t\rho+\nabla\cdot(\rho\velocity) & = 0, \label{IPM:1}\\
%\Div\velocity & = 0,\label{IPM:2}\\
%\frac{\nu}{\kappa}\velocity & = -\nabla p+\rho\mathbf{g}, \label{IPM:3}
%\end{align}
%in $\R^2\times(0,T)$,

In section 2 we revisite the notation of function spaces, pseudodifferential operators and discuss the admissible opening speeds for the mixing zone. In Section 3 we prove the key commutator estimates, in Section 4 we show how these estimates give information about the smoothing properties of our operators in the Sobolev spaces $H^s$, and in Section 5 we give the proofs of the main theorems.

\subsection*{Acknowledgments} VA, AC and DF have been supported by the Spanish Ministry of Economy under the ICMAT Severo Ochoa grant SEV2015-0554 and VA and AC also by the Europa Excelencia program ERC2018-092824. AC is partially supported by the MTM2017-89976-P and the ERC Advanced Grant 788250.  D.F  is supported by the ERC Advanced Grant  834728 and by the MTM2017-85934-C3-2-P.
The authors acknowledge Fabricio Macià for his suggestion to address this problem from the semiclassical point of view.

\section{Admissible symbol classes}
\label{s:symbol_classes}

\subsection{Symbols with limited smoothness}

We will consider the following classes of symbols. First we consider symbols having a finite number of derivatives in $L^\infty(\R_x\times \R_\xi)$.

\begin{definition}
\label{d:classes_infinity} Let $m \in \R$ and let $j,k \in \mathbb{N}_0$. A symbol $a(x,\xi)$ belongs to the class $\mathcal{M}_{j,k}^m$ if:
\begin{itemize}
\item  $a \in W^{k,\infty}_{loc}\big(\R_\xi ; W^{j,\infty}(\R_x)\big)$.
\medskip

\item Moreover,
$$
M^m_{j,k}(a) := \sup_{\alpha \leq j, \, \beta \leq k} \, \sup_{(x,\xi) \in \R^2} \, (1 + \vert \xi \vert)^{\vert \beta \vert - m} \big \vert \partial_x^\alpha \partial_\xi^\beta a(x,\xi) \big \vert < \infty.
$$
\end{itemize}
If $m = k = 0$, we denote simply $M_j(a) := M_{j,0}^0(a)$.

\end{definition}

We will also consider symbols that belong to $H^s(\R_x)$ in the $x$ variable,  while in the $\xi$ variable have a finite number of bounded derivatives.

\begin{definition}
\label{d:sobolev_in_x_spaces}
Let $m \in \R$, $k \in \mathbb{N}_0$ and $s > 1/2$. A symbol $a(x,\xi)$ belongs to the class $\mathcal{N}_{s,k}^m$ if:
\begin{itemize}
\item  $a \in W^{k,\infty}_{loc}\big(\R_\xi ; H^s(\R_x)\big)$.
\medskip

\item Moreover,
$$
N_{s,k}^m(a) := \sup_{\beta \leq k} \sup_{\xi \in \R} \, (1 + \vert \xi \vert)^{\vert \beta \vert - m}  \big \Vert \partial_\xi^\beta a(\cdot,\xi) \big \Vert_{H^s(\R)} < \infty.
$$
\end{itemize}
If $m = k = 0$, we denote simply $N_s(a) := N_{s,0}^0(a)$.
\end{definition}

Given a symbol $a(x,\xi)$, the canonical quantization $a(x,D)$ is defined acting on Schwartz functions by
\begin{equation}
\label{e:canonical_quantization}
a(x,D) f(x) := \int_{\R} e^{2\pi i x \xi} a(x,\xi) \widehat{f}(\xi) d\xi, \quad  f \in \mathscr{S}(\R),
\end{equation}
where $\widehat{f}$ denotes the Fourier transform, with the convention
$$
\widehat{f}(\xi) = \int_\R e^{-2\pi i x \xi} f(x) dx.
$$
The symbols under consideration will also depend on time $t \geq 0$, which will play the role of semiclassical parameter. The semiclassical quantization $a(x,tD)$ is defined by
\begin{equation}
\label{e:semiclassical_quantization}
a(x, tD) f(x) =  \int_{\R} e^{2\pi i x \xi} a(x,t \xi) \widehat{f}(\xi) d\xi, \quad f \in \mathscr{S}(\R).
\end{equation}

We are interested in the action of $a(x,tD)$ on the Sobolev spaces $H^s(\R)$. Since the decayment properties of $a$ in the $\xi$ variable scale in terms of the semiclassical parameter $t$,  it is usefull to include the semiclassical parameter also in the Sobolev spaces $H^s(\R)$. To this aim, we recall the following definition of semiclassical Sobolev space (see for instance \cite[Sect. 8.3, eq. (8.3.5)]{Zw12} or \cite[Sect. 2.1]{Tex07}):
\begin{definition}
Let $s \in \R$. We define
\begin{align*}
H^s_t(\R) & := \{ f \in \mathcal{D}'(\R) \, : \, (1 + it\xi)^s \, \widehat{f}(\xi) \in L^2(\R) \}, \quad t \in (0,1].
\end{align*}
For $t =1$, we have that $H^s_1(\R) = H^s(\R)$.
\end{definition}

\begin{remark}
\label{r:isometry}
Notice that the operator
$$
\begin{array}{rcl}
U_t :  H^s & \longrightarrow & H^s_t \\[0.2cm]
          u(x) & \longmapsto & U_tu(x) = t^{-1/2} u(t^{-1} x)
\end{array}
$$
is unitary and $U_t^* a(x,tD) U_t = a^t(x,D)$ with $a^t(x,\xi) = a(tx,\xi)$.
\end{remark}

From now on, we consider a function $c(x)$ belonging to the following class.
\begin{definition}
\label{d:admissible}
Let $c_1,c_2 > 0$ satisfying that
\begin{equation}
\label{e:admissible_pair}
0< c_1 < c_2 \leq 2, \quad c_1^{-1} - c_2^{-1} \leq 1.
\end{equation}
We say that a function $c : \R \to \R$ is $(c_1,c_2)$-admissible if:
\medskip

\textnormal{\textbf{(C1)}}  $c_1 < \inf_{x \in \R} c(x) \leq \sup_{x \in \R} c(x) < c_2 $.
\medskip

\noindent Moreover, $c$ satisfies at least one of the following conditions:
\medskip

\textnormal{\textbf{(C2)}} $\hspace*{0.1cm} c \in W^{N,\infty}(\R)$, for some $N \in \mathbb{N}$ satisfying that $N > 3/2 + c_1^{-1}$ and $N \geq 1 + \lceil c_1^{-1} \rceil$.
\medskip

\textnormal{\textbf{(C2')}} There exists $c_0 \in (c_1,c_2)$ such that $v_1(x) := c(x) - c_0$ and $v_2(x) := c(x)^{-1} - c_0^{-1}$ satisfy
\begin{equation*}
v = (v_1,v_2) \in H^{s}(\R) \times H^{s}(\R), \quad s > 3/2 +  c_1^{-1}, \quad s \geq 1 + \lceil c_1^{-1} \rceil.
\end{equation*}

\end{definition}

\begin{remark}
In the particular case $(c_1,c_2) = (1,2)$, for $\textbf{(C2)}$ to hold it is sufficient that $N \geq 3$, while for $\textbf{(C2')}$ it is sufficient that $s > 5/2$.
\end{remark}

We define the symbol $d_t(x,\xi) = t^{-1}d(x,t \xi)$, where $d$ is given by \eqref{e:differential_symbol}. With condition \textbf{(C2)}, one has $d \in \mathcal{M}_{N,1}^0$. Otherwise, assuming \textbf{(C2')} and following \cite{Lannes06}, we can write
$$
d(x,\xi) = \Pi_d(x,\xi) + \Sigma_d(0,\xi)
$$
for some $\Pi_d \in \mathcal{N}^0_{s,1}$ and some Fourier multiplier $\Sigma_d(0,\xi) \in \mathcal{M}^0_{\infty,1}$. Precisely,
$$
d(x,\xi) = \frac{\vert \xi \vert}{1 + (c(x) - c_0) \vert \xi \vert + c_0 \vert \xi \vert} = \frac{\vert \xi \vert}{1 + v_1(x) \vert \xi \vert + c_0 \vert \xi \vert} =: \Sigma_d(v(x),\xi),
$$
where $\Sigma_d \in \mathcal{C}^\infty(U ; W^1_{loc}(\R_\xi))$ with $U \subset \R^2$ being a neighborhood  of the image of $v$. Moreover, $\Pi_d(x,\xi) := \Sigma_d(v(x),\xi) - \Sigma_d(0,\xi)$ satisfies that $\Pi_d \in \mathcal{N}^0_{s,1}$ due to Moser's inequality, provided that $\Sigma_d$ is smooth in the image of $v$ (see for instance \cite[Prop. 3.9]{Tay10}).

We next consider the symbol
$$
p(x,\xi) := (1 + c(x) \vert \xi \vert)^{\frac{1}{c(x)}}.
$$
We will denote
\begin{equation}
m_j := c_j^{-1}, \quad j=1,2.
\end{equation}
Notice that $p \in \mathcal{M}_{N,1}^{m_1}$ provided that \textbf{(C2)} holds. Otherwise, in the case of \textbf{(C2')}, as we have done for $d$, we can write
$$
p(x,\xi) = (1 + v_1(x) \vert \xi \vert + c_0 \vert \xi \vert)^{v_2(x) + \frac{1}{c_0}} =: \Sigma_p(v(x),\xi),
$$
and $\Pi_p(x,\xi) = \Sigma_p(v(x),\xi) - \Sigma_p(0,\xi)$ satisfies that $\Pi_p \in \mathcal{N}^{m_1}_{s,1}$ provided that $\Sigma_p$ is smooth in $U$, while $\Sigma_p(0,\xi) \in \mathcal{M}_{\infty,1}^{m_1}$.  Simi\-larly, its inverse $p^{-1} \in \mathcal{M}_{N,1}^{-m_2}$ if \textbf{(C2)}, (resp. $\Pi_{p^{-1}} \in \mathcal{N}_{s,1}^{-m_2}$ and $\Sigma_{p^{-1}}(0,\xi) \in \mathcal{M}_{\infty,1}^{-m_2}$ if \textbf{(C2')}).

\subsection{Comments on possible generalizations}
Notice that
\begin{align*}
\partial_x p(x,\xi) & = \frac{c'(x)\vert \xi \vert}{c(x)}(1 + c(x) \vert \xi \vert)^{\frac{1}{c(x)}-1} - \frac{c'(x)}{c(x)^2}(1+c(x)\vert \xi \vert)^{\frac{1}{c(x)}} \log (1 + c(x) \vert \xi \vert),
\end{align*}
which means that taking derivatives of $p$ in the $x$ variable increases its growth in the $\xi$ variable by a logarithm term. This suggests that condition \textbf{(C1)} can be slightly relaxed to deal with the extremal cases of:
$$
c_1 \leq \inf_{x \in \R} c(x) \leq \sup_{x \in \R} c(x) \leq c_2.
$$
With this assumption and \textbf{(C2)}, one has that $p \in \mathcal{M}^{m_1,\rho}_{N,1}$ and $p^{-1} \in \mathcal{M}^{-m_2,\rho}_{N,1}$ for every $0 < \rho < 1$, where the spaces $\mathcal{M}^{m,\rho}_{j,k}$ are defined in terms of the seminorms
$$
M^{m,\rho}_{j,k}(a) := \sup_{\alpha \leq j, \, \beta \leq k} \, \sup_{(x,\xi) \in \R^2} \, (1 + \vert \xi \vert)^{  \beta - \rho \alpha - m} \big \vert \partial_x^\alpha \partial_\xi^\beta a(x,\xi) \big \vert < \infty.
$$
This corresponds with the symbol classes considered in \cite[Thms. 3 and 5]{Hwang87}, which could be most likely adapted to get our results in this slightly more general form. However, these extensions  do not cover  the Sobolev norms of Definition \ref{d:sobolev_in_x_spaces}, so we prefer to state our admissible hypothesis in the form of Definition \ref{d:admissible} to facilitate the reading.

The condition $c_1^{-1} - c_2^{-1} \leq 1$ in \textbf{(C1)} imposes some restrictions on the global variation of the ``width of the interphase" determined by $c(x)$  in our model problem \eqref{IVP}. This restriction appears as the global gain of regularity (corresponding with the gain of one degree of decayment in the momentum variable $\xi$ of the symbol) when performing the commutators of $p^{-1}(x,tD)$ and $d(x,tD)$, and of $p^{-1}(x,tD)$ and $p(x,tD)$, along the proof of Theorem \ref{t:main_theorem1}. This hypothesis is only needed to obtain the global energy estimate of Theorem \ref{t:main_theorem1}, but it is irrelevant to get the local smoothing properties of $p^{-1}(x,tD)$ of Theorem \ref{t:localization}.

We prove our semiclassical estimates in the one dimensional case $x \in \R$ requiered to treat equation \eqref{IVP}, but our approach also works for symbols in higher dimensions.

The condition $s \geq 1 + \lceil c_1^{-1} \rceil$ in \textbf{(C2')} allows us to avoid the use of paradifferential calculus in our commutator estimates. Extending the techniques of paradifferential calculus used in \cite{Lannes06} to remove this condition in our semiclassical setting, up to our view, seems to be non-trivial at this level of regularity.

\section{Commutator estimates}

In this section we revisit some commutator estimates obtained in \cite{CCF16} and extend them for the proofs of Theorems \ref{t:main_theorem1} and \ref{t:localization}. The main ideas come from \cite{Hwang87}.

\subsection{Preliminary lemmas}

\begin{lemma}{\cite[Lemma 5.6]{CCF16}}
\label{l:elementary}
Let $f \in L^2(\R)$, $\gamma \in H^s(\R)$ for every $s \in \R$, and define
\begin{align}
\label{e:Gamma}
\Gamma^{\pm}_{f,\gamma}(y,\eta) & := \int_\R e^{\mp 2\pi i \eta z} \gamma(y \mp z)f(z) dz.
\end{align}
Then, for every $n \geq 0$,
\begin{align}
\label{e:Gamma_1}
\Vert \partial_y^n \Gamma^{\pm}_{f,\gamma} \Vert_{L^2(\R^2)} = \Vert \gamma^{(n)} \Vert_{L^2} \Vert f \Vert_{L^2}.
\end{align}
Moreover, let $p(y,\eta) \in \mathcal{M}_{n,0}^0$, then
\begin{equation}
\label{e:Gamma_2}
\Vert \partial_y^n \big( p \, \Gamma^{\pm}_{f,\gamma} ) \Vert_{L^2(\R^2)} \leq C M_n(p) \Vert \gamma \Vert_{H^n} \Vert f \Vert_{L^2}.
\end{equation}
\end{lemma}

\begin{proof}
By Plancherel in the $y$-variable, we have
$$
\int_{\R^2} \vert \partial_y^n \Gamma^{\pm}_{f,\gamma}(y,\eta) \vert^2 dy d\eta = \int_{\R^2} \vert \widehat{\partial_y^n \Gamma^{\pm}}_{f,\gamma}(\cdot, \eta)(\lambda) \vert^2 d\lambda d\eta.
$$
On the other hand,
$$
\widehat{\partial_y^n \Gamma^{\pm}}_{f,\gamma}(\cdot, \eta)(\lambda) = \widehat{\gamma^{(n)}}(\lambda) \mathcal{F}_{\pm}f(\lambda + \eta),
$$
where $\mathcal{F}_+$ is the Fourier transform and $\mathcal{F}_-$ its inverse. By Fubini and Plancherel, we conclude \eqref{e:Gamma_1}.  To show \eqref{e:Gamma_2}, we use the product rule to write
$$
\partial_y^n \big( p(y,\eta) \Gamma^{\pm}_{f,\gamma}(y,\eta) \big) = \sum_{j=0}^n {n \choose j} \partial_y^j p(y,\eta) \, \partial_y^{n-j} \Gamma^{\pm}_{f,\gamma}(y,\eta).
$$
Hence \eqref{e:Gamma_2} follows applying \eqref{e:Gamma_1} and
$$
\Vert \partial_y^j p \, \partial_y^{n-j} \Gamma^{\pm}_{f,\gamma} \Vert_{L^2(\R^2)} \leq \Vert \partial_y^j p \Vert_{L^\infty(\R^2)} \Vert \partial_y^{n-j} \Gamma^{\pm}_{f,\gamma} \Vert_{L^2(\R^2)}.
$$
\end{proof}

\begin{lemma}{\cite[Lemma 5.9]{CCF16}}
\label{l:take_2d}
Let $Q(x,\xi)$, define
$$
A_Q(x) := \int_\R e^{2\pi i x\xi} Q(x,\xi) d\xi.
$$
Then
$$
\Vert A_Q \Vert_{L^2(\R)} \leq C \Vert (1-\partial_x) Q \Vert_{L^2(\R^2)}.
$$
\end{lemma}
\begin{proof}
We will prove the Lemma by duality. Let $g \in L^2(\R)$, we have
\begin{align*}
\int_\R A_Q(x) g(x) dx & = \int_{\R^2} e^{2\pi i x \xi} Q(x,\xi) g(x) d\xi dx  \\[0.2cm]
 & = \int_{\R^3} e^{2\pi i x (\xi + \lambda)} Q(x,\xi) \widehat{g}(\lambda) d\lambda d\xi dx.
\end{align*}
Integrating by parts in the $x$-variable, we obtain
$$
\int_\R A_Q(x) g(x) dx = \int_{\R^2} (1 - \partial_x) Q(x,\xi) \Gamma^{-}_{\widehat{g},\gamma}(\xi,x) e^{2\pi i x \cdot \xi} dx d\xi,
$$
where $\gamma$ is given by \eqref{e:special_function}. Then, applying the Cauchy-Schwartz inequality and Lemma \ref{l:elementary}, we conclude that
$$
\left \vert \int_\R A_Q(x) g(x) dx \right \vert \leq C \Vert (1-\partial_x)Q \Vert_{L^2(\R^2)} \Vert \Vert g \Vert_{L^2(\R)}.
$$
\end{proof}

\subsection{Commutator estimates}
Let $p_1$ and $p_2$ be two symbols, we define:
\begin{equation}
\label{e:semicommutator}
\mathfrak{C}(p_1,p_2) := p_1(x,D) p_2(x,D) - p_1 p_2(x,D).
\end{equation}
\begin{lemma}\cite[Thm. 5.2]{CCF16}
\label{l:castro}
Let $p_1 \in \mathcal{M}_{1,1}^0$ and $p_2 \in \mathcal{M}_{2,1}^0$. Then
\begin{equation}
\label{e:castro_estimate}
\Vert \mathfrak{C}(p_1,p_2) \Vert_{\mathcal{L}(L^2)} \lesssim M_1( \partial_\xi p_1 ) M_2( p_2) + M_1( p_1) M_1( \partial_\xi p_2 )  + M_1( \partial_\xi p_1 ) M_1( \partial_\xi p_2 ).
\end{equation}
\end{lemma}

\begin{remark}
An extended proof of Lemma \ref{l:castro} is given in \cite{CCF16}. We next rewrite the same proof in a more compact form, because some of the ideas will be used later on.
\end{remark}

\begin{proof}  We start by writing the expressions of $p_1(x,D) p_2(x,D)f$ and $p_1p_2(x,D)f$:
$$
p_1(x,D) p_2(x,D)f(x) = \int_{\R^3} e^{2\pi i (x\xi - \xi y + y \eta)} p_1(x,\xi) p_2(y,\eta) \widehat{f}(\eta) d\eta dy d\xi,
$$
while
$$
p_1p_2(x,D)f(x) = \int_{\R} e^{2\pi i x\eta} p_1(x,\eta)p_2(x,\eta) \widehat{f}(\eta) d\eta .
$$
Therefore, using that
\begin{align*}
p_1(x,\xi)p_2(y,\eta) & = \big( p_1(x,\xi)- p_1(x,\eta) \big) p_2(y,\eta) \\[0.2cm]
 & \quad + p_1(x,\eta) \big( p_2(y,\eta) - p_2(x,\eta) \big) + p_1(x,\eta) p_2(x,\eta),
\end{align*}
and the fact that, in the sense of distributions,
$$
\int_\R e^{2\pi i(x-y)\xi} d\xi = \delta(x-y),
$$
we obtain
\begin{equation}
\label{e:integral_formula_commutator}
\mathfrak{C}(p_1,p_2)f(x) = \int_{\R^3} e^{2\pi i(x\xi - \xi y + y \eta)} \big( p_1(x,\xi) - p_1(x,\eta) \big) p_2(y,\eta) \widehat{f}(\eta) d\eta dy d\xi.
\end{equation}
By the Fourier inversion formula,
$$
\mathfrak{C}(p_1,p_2)f(x) = \int_{\R^4} e^{2\pi i(x\xi - \xi y + y \eta - \eta z)}\big( p_1(x,\xi) - p_1(x,\eta) \big) p_2(y,\eta) f(z) dz d\eta dy d\xi.
$$
Using next the identities
\begin{align}
\label{e:equivalencie_1}
\frac{1}{1 + 2\pi i (y-z)} (1 + \partial_\eta) e^{2\pi i (y-z) \eta} & = e^{2\pi i (y-z)\eta}, \\[0.2cm]
\label{e:equivalencie_2}
\frac{1}{1 + 2\pi i (x-y)} (1 + \partial_\xi) e^{2\pi i (x-y) \xi} & = e^{2\pi i (x-y) \xi}, \\[0.2cm]
\label{e:equivalencie_3}
\frac{1}{1 + 2\pi i (\xi-\eta)} (1 + \partial_y) e^{2\pi i (\xi-\eta) y} & = e^{2\pi i (\xi-\eta)y},
\end{align}
we integrate by parts in $\eta$, $\xi$ and $y$ successively to get
\begin{align*}
\mathfrak{C}(p_1,p_2)f(x) & \\[0.2cm]
 & \hspace*{-2cm} = \int_{\R^3} e^{2\pi i \sigma(x,\xi,y,\eta)} \gamma(\xi-\eta) \mathcal{D}_y \big( \gamma (x-y)\Gamma^{+}_{f,\gamma}(y,\eta) \mathcal{D}_\xi  \mathcal{D}_\eta \big( (p_1(x,\xi) - p_1(x,\eta)) p_2(y,\eta) \big) \big)  d\eta dy d\xi,
\end{align*}
where $\sigma(x,\xi,y,\eta) = x \xi - \xi y + y \eta$, the differential operators $\mathcal{D}_w = 1 - \partial_w$ act on all the functions on its right (via the product rule), and
\begin{equation}
\label{e:special_function}
 \gamma(w) = \frac{1}{1 + 2\pi i w}.
\end{equation}
Expanding the derivatives by the product rule, we reach to a sum of terms of the form:
\begin{equation}
\label{e:T_term}
T_\iota f(x) =  \int_{\R^3} e^{2\pi i \sigma(x,\xi,y,\eta)} \gamma_1^\iota(\xi-\eta) \gamma_2^\iota(x-y)  a_1^\iota(x,\xi,\eta) a_2^\iota(y,\eta) \Gamma^{+}_{f,\gamma_3^\iota}(y,\eta) d\eta dy d\xi,
\end{equation}
for some explicit functions $a_1^\iota,a_2^\iota$ and some $\gamma_1^\iota, \gamma_2^\iota, \gamma_3^\iota$ given by some derivatives of the function \eqref{e:special_function}. Using Lemma \ref{l:take_2d} yields that
$$
\Vert T_\iota f \Vert_{L^2(\R)} \leq \Vert \mathcal{D}_x G_\iota \Vert_{L^2(\R_x \times \R_\xi)},
$$
where
$$
G_\iota(x,\xi) = \int_{\R^2} e^{2\pi i( \eta - \xi)y} \gamma_1^\iota(\xi-\eta) \gamma_2^\iota(x-y)  a_1^\iota(x,\xi,\eta) a_2^\iota(y,\eta) \Gamma^{+}_{f,\gamma_3^\iota}(y,\eta) d\eta dy.
$$
By the Cauchy-Schwartz inequality, we get
\begin{equation}
\label{e:after_cauchy_schwartz}
\vert \mathcal{D}_x G_\iota(x,\xi) \vert^2 \leq \Vert \gamma_1^\iota \Vert_{L^2}^2 \int_{\R} \left \vert \int_\R e^{2\pi i( \eta - \xi)y}  \mathcal{D}_x \gamma_2^\iota(x-y)  a_1^\iota(x,\xi,\eta) a_2^\iota(y,\eta) \Gamma^{+}_{f,\gamma_3^\iota}(y,\eta) dy \right \vert^2 d\eta.
\end{equation}
Expanding the derivatives in $x$, we obtain
$$
\Vert \mathcal{D}_x G_\iota \Vert^2_{L^2(\R^2)}  \lesssim \Vert \gamma_1^\iota \Vert_{L^2}^2 \Vert \mathcal{D}_x a^\iota_1 \Vert_{L^\infty(\R^3)}^2 I_\iota,
$$
where
$$
I_\iota = \int_{\R^3} \left \vert \int_\R e^{2\pi i( \eta - \xi)y}  \mathcal{D}_x \gamma_2^\iota(x-y) a_2^\iota(y,\eta) \Gamma^{+}_{f,\gamma^\iota_3}(y,\eta) dy \right \vert^2 d\eta dx d\xi.
$$
We next do Plancherel in $x$, then Fubini to integrate first with respect to $\xi$ and conclude again with Plancherel with real variable $y$ and Fourier variable $\xi$:
\begin{align*}
I_\iota & = \int_{\R^3} \vert \widehat{\mathcal{D}_x \gamma^\iota_2}(\lambda) \vert^2 \left \vert \int_\R e^{2\pi i(\eta - \xi - \lambda)y} a_2^\iota(y,\eta) \Gamma^{+}_{f,\gamma^\iota_3}(y,\eta) dy \right \vert^2 d\xi d \lambda d\eta \\[0.2cm]
 & =  \int_{\R} \vert \widehat{\mathcal{D}_x \gamma^\iota_2}(\lambda) \vert^2  \int_{\R^2}\left \vert  e^{2\pi i(\eta  - \lambda)y} a^\iota_2(y,\eta) \Gamma^{+}_{f,\gamma^\iota_3}(y,\eta) \right \vert^2 dy d \lambda d\eta \\[0.2cm]
 & \leq C \Vert \gamma_2^\iota \Vert_{H^1}^2  \int_{\R^2} \left \vert a^\iota_2(y,\eta) \Gamma^{+}_{f,\gamma_3^\iota}(y,\eta) \right \vert^2 dy  d\eta.
\end{align*}
Finally, using Lemma \ref{l:elementary}, we conclude that
$$
I_\iota \leq C \Vert \gamma^\iota_2 \Vert_{H^1}^2 M_0( a^\iota_2 )^2 \Vert f \Vert^2_{L^2}.
$$

Notice also that, in some of the terms $T_\iota$, it appears
$$
a^\iota_1(x,\xi,\eta) =  p_1(x,\xi) - p_1(x,\eta).
$$
In order to estimate this factor in terms of $\partial_\xi p_1$, it is necessary to integrate by parts one more time in $y$, using the identity
\begin{equation}
\label{e:identity_w_r_t_y}
\frac{1}{2\pi i (\xi-\eta)} \partial_y \, e^{2\pi i (\xi-\eta) y}  = e^{2\pi i (\xi-\eta)y},
\end{equation}
to obtain a new function
\begin{align}
\label{e:a_1_dagger}
a^\dagger_1(x,\xi,\eta) = \frac{p_1(x,\xi) - p_1(x,\eta)}{2\pi i (\xi - \eta)},
\end{align}
which, by the mean value theorem, satisfies
\begin{align*}
 \Vert \mathcal{D}_x a^\dagger_1 \Vert_{L^\infty(\R^3)} & \leq M_1(\partial_\xi p_1).
\end{align*}

Taking into account all the derivatives of the symbols $p_1$ and $p_2$ we have performed in each term $T_\iota$, we obtain
$$
\Vert \mathfrak{C}(p_1,p_2) f \Vert_{L^2} \lesssim \Big( M_1( \partial_\xi p_1) M_2( p_2 ) + M_1( p_1) M_1( \partial_\xi p_2)  + M_1( \partial_\xi p_1 ) M_1( \partial_\xi p_2 ) \Big) \Vert f \Vert_{L^2},
$$
as we wanted to prove.
\end{proof}

We next explain how Lemma \ref{l:castro} applies to the semiclassical framework. To this aim, let us introduce first the following notation for the semiclassical non-principal part of the composition of  two semiclassical operators $p_1(x,tD)$, and $p_2(x,tD)$:
\begin{equation}
\label{e:approximate_inverse}
\mathfrak{C}_t(p_1,p_2) := p_1(x,tD) p_2(x,tD)- p_1 p_2(x,tD), \quad t \in (0,1].
\end{equation}
We  will also consider a localized version of $\mathfrak{C}_t(p_1,p_2)$ near the diagonal of $\R_\xi \times \R_\eta$. Given $\varphi \in \mathcal{C}_c^\infty(\R)$, we define:
\begin{equation}
\label{e:localized_semiclassical_commutator}
\mathfrak{C}_{t,\varphi}(p_1,p_2)f(x) := \int_{\R^3} e^{2\pi i (x\xi - \xi y + y\eta)} \varphi\big( t(\xi - \eta) \big) \big( p_1(x,t\xi) - p_1(x,t\eta) \big) p_2(y,t\eta) \widehat{f}(\eta) d\eta dy d\xi.
\end{equation}
Let
\begin{equation}
\label{e:convex_hull}
I_\varphi := \operatorname{Conv}\big(  \supp \varphi \cup \{ 0 \} \big)
\end{equation}
be the closed interval obtained as the convex hull of $\supp \varphi \cup \{ 0  \}$, and let $\vert I_\varphi \vert$ be the Lebesgue measure of $I_\varphi$. The following holds:
\begin{corol}
\label{c:localized_castro}
Let $m_1, m_2 \in \R$ such that $ m_1 + m_2 - 1 \leq 0$.  Set $\mu = \max \{ m_1,m_2,0 \}$. Let $p_1 \in \mathcal{M}^{m_1}_{1,1}$ and $p_2 \in \mathcal{M}^{m_2}_{2,1}$. Then, for every $\varphi \in \mathcal{C}_c^\infty(\R)$,
\begin{align}
\label{e:localized_semiclassical_estimate}
\Vert \mathfrak{C}_{t,\varphi}(p_1,p_2) \Vert_{\mathcal{L}(L^2)} \lesssim t \, \Vert \varphi \Vert_{W^{1,\infty}(\R)} \vert I_\varphi \vert^{\mu} \, \mathfrak{M}(p_1,p_2),
\end{align}
where
$$
\mathfrak{M}(p_1,p_2) =  M_{1,0}^{m_1-1}( \partial_\xi p_1) M_{2,0}^{m_2}( p_2 ) + M_{1,0}^{m_1}( p_1) M_{1,0}^{m_2-1}( \partial_\xi p_2)  + M_{1,0}^{m_1-1}( \partial_\xi p_1 ) M_{1,0}^{m_2-1}( \partial_\xi p_2).
$$
\end{corol}

\begin{remark}
The presence of derivatives in $\xi$ in all the terms in the right-hand-side of \eqref{e:castro_estimate} allow us to bring the factor $t$ in the semiclassical estimate \eqref{e:localized_semiclassical_estimate}.
\end{remark}

\begin{proof}
The proof mimics the one of Lemma \ref{l:castro}, but in this case, we replace $p_1(x,\xi) - p_1(x,\eta)$ in \eqref{e:integral_formula_commutator} by
\begin{equation}
\label{e:new_symbol}
a_1(x,\xi,\eta) = \varphi(\xi - \eta) \big( p_1(x,\xi) - p_1(x,\eta) \big) (1 + 2\pi i \eta)^{m_2},
\end{equation}
and  $p_2(y,\eta)$ by $p_2^\dagger(y,\eta) = (1+ 2\pi i \eta)^{-m_2}p_2(y,\eta)$. Using next the identities \eqref{e:equivalencie_1}, \eqref{e:equivalencie_2} and \eqref{e:equivalencie_3}, we integrate by parts in $\eta$, $\xi$ and $y$ successively to get
\begin{align*}
\mathfrak{C}_{t,\varphi}(p_1,p_2)f(x)  & \\[0.2cm]
 & \hspace*{-1cm} = \int_{\R^3} e^{2\pi i \sigma(x,\xi,y,\eta)} \gamma(\xi-\eta) \mathcal{D}_y \big( \gamma (x-y)\Gamma^{+}_{f,\gamma}(y,\eta) \mathcal{D}_\xi  \mathcal{D}_\eta \big( a_1(x,t\xi,t\eta) p_2^\dagger(y,t\eta) \big) \big)  d\eta dy d\xi,
\end{align*}
where $\sigma(x,\xi,y,\eta) = x \xi - \xi y + y \eta$. Expanding the derivatives by the product rule, we reach again to a sum of terms of the form \eqref{e:T_term} after the obvious substitutions. In particular, when no derivatives in $\xi$ nor $\eta$ are performed in $a_1$, a further use of integration by parts in the $y$ variable, as we did to obtain \eqref{e:a_1_dagger}, allow us to replace $a_1^\iota = a_1$ by
\begin{align*}
a_1^{\dagger}(x,t\xi,t\eta) & = \frac{\varphi\big((t(\xi - \eta)\big) \big(p_1(x,t\xi)-p_1(x,t\eta)\big) (1+ 2\pi i t\eta)^{m_2}}{2\pi i(\xi - \eta)} \\[0.2cm]
 & = t\cdot  \frac{\varphi\big((t(\xi - \eta)\big) \big(p_1(x,t\xi)-p_1(x,t\eta)\big) (1+ 2\pi i t\eta)^{m_2}}{2\pi i t(\xi - \eta)}.
\end{align*}
We then estimate each of the terms $J_\iota$ obtained similarly as we did in the proof of Lemma \ref{l:castro}. Here we only remark the main differences and changes required in this case, which appear only when bounding the $L^\infty$ norms of $a_1^\iota$ and $a_2^\iota$. In fact, it is sufficient to indicate how the term $J_\iota$ involving $a_1^\dagger$ and $p_2^\dagger$ is managed, since the others can be bounded in a completely analogous way.

We consider the set
\begin{equation}
\label{e:convex_set}
\Omega_\varphi := \{ (\xi, \eta) \in \R^2 \, : \, \xi - \eta \in I_\varphi \},
\end{equation}
and we use that $m_1 + m_2 -1 \leq 0$ and the mean-value theorem to get
\begin{align*}
 \sup_{(x,\xi,\eta) \in \R^3} \vert \mathcal{D}_x a_1^{\dagger}(x,t\xi,t\eta) \vert &  = t  \cdot \sup_{(x,\xi,\eta) \in \R^3}  \frac{ \vert  \varphi\big( t(\xi - \eta)\big) \mathcal{D}_x \big( p_1(x,t\xi) - p_1(x,t\eta) \big)(1 + 2\pi i t \eta)^{m_2} \vert}{\vert t(\xi - \eta) \vert}\\[0.2cm]
 & = t  \cdot \sup_{(x,\xi,\eta) \in \R^3}  \frac{ \vert  \varphi (\xi - \eta) \mathcal{D}_x \big( p_1(x,\xi) - p_1(x,\eta) \big)(1 + 2\pi i  \eta)^{m_2} \vert}{\vert \xi - \eta \vert}\\[0.2cm]
 & \leq t \, \Vert \varphi \Vert_{L^\infty}  \sup_{x \in \R} \sup_{ (\xi, \eta) \in \Omega_\varphi} \vert \mathcal{D}_x \partial_\xi p_1(x,\xi) \vert \vert 1+ 2 \pi i \eta \vert^{m_2} \\[0.2cm]
 & \leq t \,M_{1,0}^{m_1 - 1}(\partial_\xi p_1) \sup_{(\xi,\eta) \in \Omega_\varphi} (1+ \vert \xi \vert)^{m_1-1} (1+ \vert \eta \vert)^{m_2} \\[0.2cm]
 &  \lesssim t \, \vert I_\varphi \vert^{\mu} M_{1,0}^{m_1 - 1}(\partial_\xi p_1).
\end{align*}
Moreover, for every $0 \leq \alpha \leq 2$,
\begin{equation}
\label{e:sup_for_p_2}
\sup_{(y,\eta) \in \R^2} \vert \partial_y^\alpha p_2^\dagger(y,t\eta) \vert \leq M_{2,0}^{m_2}(p_2).
\end{equation}
These and analogous estimates, depending on whether the derivatives $\partial_\xi$ and $\partial_\eta$ act on the factors $\varphi$, $p_1$ or $p_2$, together with the ones given in the proof of Lemma \ref{l:castro}, suffice to bound all the terms $T_\iota$ .
\end{proof}

We next deal with symbols in the classes $\mathcal{N}^0_{s,1}$. Since we already have $L^2$-decay in the $x$ variable, we do not need to integrate by parts in the momentum variables $\xi,\eta$. This simplifies the proof.
\begin{lemma}
\label{c:commutator_in_sobolev}
Let $p_1 \in \mathcal{N}_{1,1}^0$ and $p_2 \in \mathcal{N}_{2,0}^0$.
Then
$$
\Vert \mathfrak{C}(p_1,p_2) \Vert_{\mathcal{L}(L^2)} \lesssim N_1( \partial_\xi p_1 ) N_2( p_2 ).
$$
\end{lemma}

\begin{proof}
We now have
\begin{align*}
\mathfrak{C}(p_1,p_2)f(x) = \int_{\R^3} e^{2\pi i \sigma(x,\xi,y,\eta)} \gamma(\xi-\eta)  \big(p_1(x,\xi) - p_1(x,\eta) \big) \mathcal{D}_y p_2(y,\eta) \widehat{f}(\eta)  d\eta dy d\xi.
\end{align*}
We then integrate by parts one more time in $y$, using the identity
$$
\frac{1}{2\pi i (\xi-\eta)} \partial_y \, e^{2\pi i (\xi-\eta) y}  = e^{2\pi i (\xi-\eta)y},
$$
to obtain
\begin{align*}
\mathfrak{C}(p_1,p_2)f(x) = \int_{\R^3} e^{2\pi i \sigma(x,\xi,y,\eta)} \gamma(\xi-\eta)  \left( \frac{p_1(x,\xi) - p_1(x,\eta)}{2\pi i (\xi - \eta)} \right) \partial_y \mathcal{D}_y p_2(y,\eta) \widehat{f}(\eta)  d\eta dy d\xi.
\end{align*}
Considering
\begin{equation}
\label{e:to_minkowski}
a^\dagger_1(x,\xi,\eta) := \frac{p_1(x,\xi) - p_1(x,\eta)}{2\pi i (\xi - \eta)} = \frac{1}{2\pi i} \int_0^1 \partial_\xi p_1(x,\eta + s(\xi - \eta)) ds,
\end{equation}
%satisfies, by the mean value theorem,
%\begin{align}
%\label{e:mean_value1}
%\sup_{(\xi,\eta) \in \R^2} \big \Vert a^\dagger_1(\cdot, \xi,\eta) \big \Vert_{L^2(\R)} & \leq N_0(\partial_\xi p_1), \\[0.2cm]
%\label{e:mean_value2}
%\sup_{(\xi,\eta) \in \R^2}\big \Vert  \partial_x a^\dagger_1(\cdot, \xi,\eta) \big \Vert_{L^2(\R)}  & \leq N_1( \partial_\xi p_1).
%\end{align}
and using Lemma \ref{l:take_2d} yields that
$$
\Vert \mathfrak{C}(p_1,p_2)f \Vert_{L^2} \leq \Vert \mathcal{D}_x G \Vert_{L^2(\R_x \times \R_\xi)},
$$
where
$$
G(x,\xi) = \int_{\R^2} e^{2\pi i (-\xi y + y \eta)} \gamma(\xi-\eta) a^\dagger_1(x,\xi,\eta) \partial_y \mathcal{D}_y p_2(y,\eta) \widehat{f}(\eta)  d\eta dy.
$$
The end of the proof follows by similar arguments of those of Lemma \ref{l:castro}. By the Cauchy-Schwartz inequality, we have
\begin{equation}
\vert \mathcal{D}_x G(x,\xi) \vert^2 \leq \Vert \gamma \Vert_{L^2}^2 \int_{\R} \left \vert  \mathcal{D}_x a^\dagger_1(x,\xi,\eta) \int_\R e^{2\pi i( \eta - \xi)y}  \partial_y \mathcal{D}_y p_2(y,\eta) \widehat{f}(\eta) dy \right \vert^2 d\eta.
\end{equation}
Moreover, by the Minkowski integral inquality,
\begin{align*}
\int_\R  \big \vert \mathcal{D}_x a_1^\dagger(x,\xi,\eta) \big \vert^2 dx & \leq \int_\R \left \vert \int_0^1 \mathcal{D}_x \partial_\xi p_1 \big(x, \eta + s (\xi - \eta) \big) ds \right \vert^2 dx \\[0.2cm]
  & \leq \left( \int_0^1 \left( \int_\R \big \vert \mathcal{D}_x \partial_\xi p_1\big(x, \eta + s(\xi - \eta) \big) \big \vert^2 dx \right)^{1/2}ds \right)^2 \\[0.2cm]
 & \leq \sup_\xi \int_\R \vert \mathcal{D}_x \partial_\xi p_1(x,\xi) \vert^2 dx \\[0.2cm]
 & = N_1(\partial_\xi p_1)^2.
\end{align*}
Hence, using this and \eqref{e:to_minkowski}, we obtain
$$
\Vert \mathcal{D}_x G \Vert^2_{L^2(\R^2)}  \lesssim \Vert \gamma \Vert_{L^2}^2 N_1( \partial_\xi p_1)^2 \int_{\R^2} \left \vert \int_\R  e^{2\pi i( \eta - \xi)y}  \partial_y \mathcal{D}_y p_2(y,\eta) \widehat{f}(\eta) dy \right \vert^2  d\xi d\eta.
$$
Finally, using Plancherel in the $\xi$-variable, we conclude that
\begin{align*}
\Vert \mathcal{D}_x G \Vert^2_{L^2(\R^2)} & \lesssim \Vert \gamma \Vert_{L^2}^2 N_1( \partial_\xi p_1)^2 \int_{\R^2} \left \vert   \partial_y \mathcal{D}_y p_2(y,\eta) \widehat{f}(\eta)  \right \vert^2  dy d\eta \\[0.2cm]
 & \lesssim  \Vert \gamma \Vert_{L^2}^2  N_1( \partial_\xi p_1)^2  N_{2}(p_2)^2 \Vert f \Vert_{L^2}^2.
\end{align*}
\end{proof}
The following corollary is a semiclassical and localized version of Lemma \ref{c:commutator_in_sobolev}.
\begin{corol}
\label{c:localized_commutator_in_sobolev}
Let $p_1 \in \mathcal{N}_{1,1}^{m_1}$ and $p_2 \in \mathcal{N}_{2,0}^{m_2}$ with $m_1+ m_2 - 1 \leq 0$. Set $\mu = \max \{m_1,m_2,0 \}$. Then, for every $\varphi \in \mathcal{C}_c^\infty(\R)$,
$$
\Vert \mathfrak{C}_{t,\varphi}(p_1,p_2) \Vert_{\mathcal{L}(L^2)} \lesssim t  \Vert \varphi \Vert_{W^{1,\infty}(\R)}\vert I_\varphi \vert^{\mu} N_{1,0}^{m_1-1}( \partial_\xi p_1 ) N_{2,0}^{m_2}( p_2 ).
$$
\end{corol}

Combining the two previous lemmas, we also have:

\begin{lemma}
\label{c:mixed_commutator}
Let $p_1 \in \mathcal{M}_{1,1}^0$ and $p_2 \in \mathcal{N}_{2,0}^0$.
Then
$$
\Vert \mathfrak{C}(p_1,p_2) \Vert_{\mathcal{L}(L^2)} \lesssim M_1( \partial_\xi p_1 ) N_2( p_2 ).
$$
Similarly, let $p_1 \in \mathcal{N}_{1,1}^0$ and $p_2 \in \mathcal{M}_{2,1}^0$.
Then
$$
\Vert \mathfrak{C}(p_1,p_2) \Vert_{\mathcal{L}(L^2)} \lesssim  N_1( \partial_\xi p_1 ) M_2( p_2 ) + N_1(p_1) M_1 (\partial_\xi p_2).
$$
\end{lemma}

\begin{corol}
\label{c:localized_mixed_commutator}
Let $p_1 \in \mathcal{M}_{1,1}^{m_1}$ and $p_2 \in \mathcal{N}_{2,0}^{m_2}$ with $m_1+ m_2 - 1 \leq 0$. Set $\mu = \max \{m_1,m_2,0 \}$. Then, for every $\varphi \in \mathcal{C}_c^\infty(\R)$,
$$
\Vert \mathfrak{C}_{t,\varphi}(p_1,p_2) \Vert_{\mathcal{L}(L^2)} \lesssim t \Vert \varphi \Vert_{W^{1,\infty}(\R)} \vert I_\varphi \vert^{\mu} M_{1,0}^{m_1-1}( \partial_\xi p_1 ) N_{2,0}^{m_2}( p_2 ).
$$
Similarly, let $p_1 \in \mathcal{N}_{1,1}^{m_1}$ and $p_2 \in \mathcal{M}_{2,1}^{m_2}$.
Then
$$
\Vert \mathfrak{C}_{t,\varphi}(p_1,p_2) \Vert_{\mathcal{L}(L^2)}  \lesssim t \Vert \varphi \Vert_{W^{1,\infty}(\R)} \vert I_\varphi \vert^{\mu} \big( N_{1,0}^{m_1-1}( \partial_\xi p_1 ) M_{2,0}^{m_2}( p_2 ) + N_{1,0}^{m_1}(p_1) M_{1,0}^{m_2-1} (\partial_\xi p_2) \big).
$$
\end{corol}

We next improve the previous lemmas when the supports of $p_1$ and $p_2$ in the $\xi$ variable are disjoint. To this aim, let us define, for any $p(x,\xi)$,
$$
\supp_\xi p := \bigcup_{x \in \R} \supp p(x,\cdot).
$$
\begin{lemma}
\label{l:disjoint_supports} Let $p_1 \in \mathcal{M}_{1,1}^0$ and $p_2 \in \mathcal{M}_{N,1}^0$ with $N \geq 2$. Assume that
$$
d := \operatorname{dist} \big( \supp_\xi p_1 , \; \supp_\xi p_2 \big) > 0.
$$
Then
$$
\Vert \mathfrak{C}(p_1,p_2) \Vert_{\mathcal{L}(L^2)} \lesssim d^{-(N-3/2)}  \big(M_1( \partial_\xi p_1) M_{N}(p_2) + M_1(p_1)  M_{N-1}( \partial_\xi p_2 )  + M_1( \partial_\xi p_1) M_{N-1}( \partial_\xi p_2) \big).
$$
\end{lemma}

\begin{proof}
The proof follows the same lines of the one of Lemma \ref{l:castro}. As in that case, we write again the expression
\begin{align*}
\mathfrak{C}(p_1,p_2)f(x) & \\[0.2cm]
 & \hspace*{-1.9cm} = \int_{\R^3} e^{2\pi i \sigma(x,\xi,y,\eta)} \gamma(\xi-\eta) \mathcal{D}_y \big( \gamma (x-y) \Gamma^{+}_{f,\gamma}(y,\eta) \mathcal{D}_\xi  \mathcal{D}_\eta \big( (p_1(x,\xi) - p_1(x,\eta))  p_2(y,\eta) \big) \big)  d\eta dy d\xi,
\end{align*}
where $\sigma(x,\xi,y,\eta) = x\xi - \xi y + y \eta$, and the differential operators $\mathcal{D}_w$ act on all the functions on its right.
We next do $(N-1)$-integrations by parts in the $y$-variable, using the identity
$$
\frac{1}{2\pi i (\xi-\eta)} \partial_y \, e^{2\pi i (\xi-\eta) y}  = e^{2\pi i (\xi-\eta)y},
$$
to bring a factor $(2\pi i(\xi - \eta))^{-(N-1)}$. Using the definition of $d > 0$, we observe that
\begin{equation}
\label{e:more_decayment}
\left \vert \frac{1}{(2\pi i (\xi - \eta))} \right \vert \leq \frac{1}{d}.
\end{equation}
Observe also that in this case, the use of Cauchy-Schwartz as before \eqref{e:after_cauchy_schwartz} allows us to obtain
\begin{align*}
\vert \mathcal{D}_x G_\iota(x,\xi) \vert^2 & \\[0.2cm]
 & \hspace*{-1.5cm} = \left \vert \int_{\vert \eta - \xi \vert \geq d} \int_\R e^{2\pi i( \eta - \xi)y} \gamma^\iota_1(\xi-\eta) \gamma^\iota_2(x-y)  \frac{ a^\iota_1(x,\xi,\eta)}{(2\pi i(\xi - \eta))^{N-2}} a_2^\iota(y,\eta) \Gamma^{+}_{f,\gamma^\iota_3}(y,\eta) dy d\eta \right \vert^2 \\[0.2cm]
& \hspace*{-1.5cm} \leq \Vert \gamma^\iota_1 \Vert_{L^2(\mho_d)}^2 \int_{\vert \eta - \xi \vert \geq d} \left \vert \int_\R e^{2\pi i( \eta - \xi)y}  \mathcal{D}_x \gamma^\iota_2(x-y)  \frac{ a^\iota_1(x,\xi,\eta)}{(2\pi i(\xi - \eta))^{N-2}} a_2^\iota(y,\eta) \Gamma^{+}_{f,\gamma^\iota_3}(y,\eta) dy \right \vert^2 d\eta \\[0.2cm]
& \hspace*{-1.5cm} \leq \Vert \gamma^\iota_1 \Vert_{L^2(\mho_d)}^2 \sup_{\vert \xi - \eta \vert \geq d} \left \Vert \frac{\mathcal{D}_x a^\iota_1(\cdot,\xi,\eta)}{(2\pi i(\xi - \eta))^{N-2}} \right \Vert_{L^\infty(\R_x)} J_\iota(x,\xi),
\end{align*}
where $\gamma^\iota_1 = \gamma$ is given by \eqref{e:special_function}, $\mho_d(\eta) = \{ \eta \in \R \, : \, \vert \eta \vert \geq d \}$, and
$$
J_\iota(x,\xi) = \int_{\R} \left \vert \int_\R e^{2\pi i( \eta - \xi)y}  \mathcal{D}_x \gamma^\iota_2(x-y)  a^\iota_2(y,\eta) \Gamma^{+}_{f,\gamma^\iota_3}(y,\eta) dy \right \vert^2 d\eta.
$$
Moreover, by \eqref{e:more_decayment},
$$
\sup_{\vert \xi - \eta \vert \geq d} \left \Vert \frac{\mathcal{D}_x a^\iota_1(\cdot,\xi,\eta)}{(2\pi i(\xi - \eta))^{N-2}} \right \Vert_{L^\infty(\R_x)} \leq d^{-(N-2)} \Vert \mathcal{D}_x a_1 \Vert_{L^\infty(\R^3)} \leq d^{-(N-2)}M_1(\partial_\xi p_1),
$$
and
\begin{equation}
\label{e:extra_decayment}
\Vert \gamma_1 \Vert_{L^2 (\mho_d)} = \left( \int_{\vert \eta \vert \geq d} \frac{d\eta}{1 + \eta^2}  \right)^{\frac{1}{2}} = \big(\pi - 2 \operatorname{arctan} (d) \big)^{\frac{1}{2}} = O\left( d^{-\frac{1}{2}}  \right), \quad \text{as } d \to \infty.
\end{equation}
The rest of the proof mimics the proof of Lemma \ref{l:castro}.
\end{proof}

From this, we obtain the following corollary which is a semiclassical and localized version of Lemma \ref{l:disjoint_supports}:

\begin{corol}
\label{c:localized_disjoint_supports}
Let $N \geq 2$. Let $p_1 \in \mathcal{M}_{1,1}^{m_1}$ and $p_2 \in \mathcal{M}_{N,1}^{m_2}$ with $m_1+m_2 -1 \leq 0$. Set $\mu = \max \{m_1,m_2,0 \}$. Let $\varphi \in \mathcal{C}_c^\infty(\R)$, assume that
$$
d := \operatorname{dist} \big( 0, \operatorname{supp} \varphi \big) > 0.
$$
Then
\begin{align*}
\Vert \mathfrak{C}_{t,\varphi}(p_1,p_2) \Vert_{\mathcal{L}(L^2)} \lesssim t^{N-1}  \, d^{-(N-\frac{3}{2})} \Vert \varphi \Vert_{W^{1,\infty}(\R)} \vert I_\varphi \vert^{\mu} \, \mathfrak{M}_N(p_1,p_2),
\end{align*}
where
$$
  \mathfrak{M}_N(p_1,p_2) = M_{1,0}^{m_1-1}( \partial_\xi p_1) M^{m_2}_{N,0}(p_2) + M_{1,0}^{m_1}(p_1)  M_{N-1,0}^{m_2-2}( \partial_\xi p_2 )  + M_{1,0}^{m_1-1}( \partial_\xi p_1) M_{N-1,0}^{m_2-1}( \partial_\xi p_2).
$$
\end{corol}

\begin{proof}
The proof mimics the one of Lemma \ref{l:disjoint_supports}, but with some changes analogous to those referred in the proof of Corollary \ref{c:localized_castro} with respect to the terms $J_\iota$. To highlight the required changes, we consider again $a_1$ given by \eqref{e:new_symbol} and use identity \eqref{e:identity_w_r_t_y} and integration by parts with respect to the the $y$ variable to replace $a_1=a_1^\iota$ by
$$
a_1^{\dagger}(x,t\xi,t\eta)  = t \cdot \frac{\varphi\big((t(\xi - \eta)\big) \big(p_1(x,t\xi)-p_1(x,t\eta)\big) (1+ 2\pi i t\eta)^{m_2}}{2\pi i t(\xi - \eta)}.
$$
To bound the term $J_\iota$ involving this symbol, we observe that
\begin{align*}
 \sup_{(x,\xi,\eta) \in \R^3} \left \vert \frac{\mathcal{D}_x a_1^{\dagger}(x,t\xi,t\eta)}{(2\pi i (\xi-\eta))^{N-2}} \right \vert &  \lesssim t^{N-1}  \sup_{(x,\xi,\eta) \in \R^3}  \frac{ \vert  \varphi\big( t(\xi - \eta)\big) \mathcal{D}_x \big( p_1(x,t\xi) - p_1(x,t\eta) \big)(1 + 2\pi i t \eta)^{m_2} \vert}{\vert t(\xi - \eta) \vert^{N-1}}\\[0.2cm]
 & \lesssim \frac{t^{N-1}}{d^{N-2}}  \sup_{(x,\xi,\eta) \in \R^3}  \frac{ \vert  \varphi (\xi - \eta) \mathcal{D}_x \big( p_1(x,\xi) - p_1(x,\eta) \big)(1 + 2\pi i  \eta)^{m_2} \vert}{\vert \xi - \eta \vert}\\[0.2cm]
 & \leq t^{N-1} d^{-(N-2)} \, \Vert \varphi \Vert_{L^\infty}  \sup_{x \in \R} \sup_{ (\xi, \eta) \in \Omega_\varphi} \vert \mathcal{D}_x \partial_\xi p_1(x,\xi) \vert \vert 1+ 2 \pi i \eta \vert^{m_2} \\[0.2cm]
 & \leq t^{N-1} d^{-(N-2)} \,M_{1,0}^{m_1 - 1}(\partial_\xi p_1) \sup_{(\xi,\eta) \in \Omega_\varphi} (1+ \vert \xi \vert)^{m_1-1} (1+ \vert \eta \vert)^{m_2} \\[0.2cm]
 &  \lesssim t^{N-1} d^{-(N-2)} \, \vert I_\varphi \vert^{\mu} M_{1,0}^{m_1 - 1}(\partial_\xi p_1).
\end{align*}
Taking into account  \eqref{e:sup_for_p_2},  \eqref{e:extra_decayment}, and the rest of estimates of the proof of Lemma \ref{l:castro}, we can manage all the terms $T_\iota$, and the result follows.
\end{proof}

The following lemma extends the previous one allowing fractional derivatives of $p_2$ in the $x$ variable.
\begin{lemma}
\label{c:disjoint_in_sobolev}
 Let $p_1 \in \mathcal{N}_{1,1}^0$ and $p_2 \in \mathcal{N}_{s,0}^0$ with $s > 3/2$. Assume that
$$
d = \operatorname{dist} \big( \supp_\xi p_1 , \; \supp_\xi p_2 \big) > 0.
$$
Then
$$
\Vert \mathfrak{C}(p_1,p_2) \Vert_{\mathcal{L}(L^2)} \lesssim  d^{-(s - 3/2)} N_1( \partial_\xi p_1) N_{s}(p_2) .
$$
\end{lemma}

\begin{proof}
The proof is analogous to the one of Lemma \ref{c:commutator_in_sobolev}. We have
\begin{align*}
\mathfrak{C}(p_1,p_2)f(x) = \int_{\R^3} e^{2\pi i \sigma(x,\xi,y,\eta)} \gamma(\xi-\eta)  \big(p_1(x,\xi) - p_1(x,\eta) \big) \mathcal{D}_y p_2(y,\eta) \widehat{f}(\eta)  d\eta dy d\xi.
\end{align*}
We then integrate by parts one more time in $y$, using the identity
$$
\frac{1}{2\pi i (\xi-\eta)} \partial_y \, e^{2\pi i (\xi-\eta) y}  = e^{2\pi i (\xi-\eta)y},
$$
to obtain
\begin{align*}
\mathfrak{C}(p_1,p_2)f(x) = \int_{\R^3} e^{2\pi i \sigma(x,\xi,y,\eta)} \gamma(\xi-\eta)  \left( \frac{p_1(x,\xi) - p_1(x,\eta)}{2\pi i (\xi - \eta)} \right) \partial_y \mathcal{D}_y p_2(y,\eta) \widehat{f}(\eta)  d\eta dy d\xi.
\end{align*}
Moreover, we also have
\begin{align*}
\mathfrak{C}(p_1,p_2)f(x) & \\[0.2cm]
 & \hspace*{-1.5cm} = \int_{\R^3} e^{2\pi i \sigma(x,\xi,y,\eta)} \gamma_s(\xi-\eta)  \left( \frac{p_1(x,\xi) - p_1(x,\eta)}{2\pi i (\xi - \eta)} \right)  \partial_y \mathcal{D}^{s-1}_y p_2(y,\eta) \widehat{f}(\eta)  d\eta dy d\xi,
\end{align*}
where
\begin{equation}
\label{e:gamma_s}
\gamma_s( \xi - \eta) := \frac{1}{(1 + 2\pi i (\xi - \eta))^{s-1}}.
\end{equation}
Therefore, it is sufficient to use the fact that
\begin{equation}
\label{e:good_s_estimate}
\Vert \gamma_s \Vert_{L^2 (\mho_d)} = \left( \int_{\vert \eta \vert \geq d} \frac{d\eta}{(1 + (2\pi \eta)^2)^{s-1}}  \right)^{\frac{1}{2}} = O\left( d^{-(s-\frac{3}{2})}  \right),
\end{equation}
and the rest of estimates given in the proof of Lemma \ref{c:commutator_in_sobolev}.
\end{proof}

\begin{corol}
\label{c:localized_disjoint_in_sobolev}
Let $s > 3/2$. Let $p_1 \in \mathcal{N}_{1,1}^{m_1}$ and $p_2 \in \mathcal{N}_{s,0}^{m_2}$ with $m_1+m_2-1 \leq 0$. Set $\mu = \max \{m_1,m_2,0 \}$. Given $\varphi \in \mathcal{C}_c^\infty(\R)$, assume that
$$
d = \operatorname{dist} \big( 0, \supp \varphi) > 0.
$$
Then
$$
\Vert \mathfrak{C}_{t,\varphi}(p_1,p_2) \Vert_{\mathcal{L}(L^2)} \lesssim t^{s-1/2} \, d^{-(s - 3/2)} \Vert \varphi \Vert_{W^{1,\infty}(\R)} \vert I_\varphi \vert^{\mu} N_{1,0}^{m_1-1}( \partial_\xi p_1) N_{s,0}^{m_2}(p_2) .
$$
\end{corol}

\begin{lemma}
\label{c:disjoint_mixed}
Let $p_1 \in \mathcal{M}_{1,1}^0$ and $p_2 \in \mathcal{N}_{s,0}^0$ with $s > 3/2$. Assume that
$$
d = \operatorname{dist} \big( \supp_\xi p_1 , \; \supp_\xi p_2 \big) > 0.
$$
Then
$$
\Vert \mathfrak{C}(p_1,p_2) \Vert_{\mathcal{L}(L^2)} \lesssim d^{-(s-3/2)}M_1( \partial_\xi p_1 ) N_{s}( p_2 ).
$$
\end{lemma}

\begin{proof}
The proof is similar to the one before, but we need to use integration by parts in the $\xi$ variable to obtain decayment in the $x$ variable (since now $p_1$ is only bounded in this variable). We have
\begin{align*}
\mathfrak{C}(p_1,p_2)f(x) = \int_{\R^3} e^{2\pi i \sigma(x,\xi,y,\eta)} \gamma(x-y) \mathcal{D}_\xi Q_s(x,\xi,\eta)  \partial_y \mathcal{D}_y^{s-1} p_2(y,\eta) \widehat{f}(\eta)  \, d\eta dy d\xi,
\end{align*}
where
$$
Q_s(x,\xi,\eta) = \gamma_s(\xi-\eta)  \left( \frac{p_1(x,\xi) - p_1(x,\eta)}{2\pi i (\xi - \eta)} \right),
$$
which is differentiable in the $\xi$ variable provided that $\vert \xi - \eta \vert \geq d$. Precisely,
\begin{align*}
\partial_\xi  Q_s(x,\xi,\eta) & = \partial_\xi \gamma_s(\xi-\eta)  \left( \frac{p_1(x,\xi) - p_1(x,\eta)}{2\pi i (\xi - \eta)} \right) - \gamma_s(\xi-\eta)  \left( \frac{p_1(x,\xi) - p_1(x,\eta)}{2\pi i(\xi - \eta)^2} \right) \\[0.2cm]
 & \quad + \gamma_s(\xi-\eta) \frac{\partial_\xi p_1(x,\xi)}{2\pi i (\xi - \eta)}.
\end{align*}
Then, using the mean-value theorem, one has
\begin{equation}
\label{e:mean_value_with_s}
\sup_{\vert \xi - \eta \vert \geq d} \vert \mathcal{D}_x \mathcal{D}_\xi Q_s(x,\xi,\eta) \vert \lesssim d^{-(s-2)}M_1(\partial_\xi p_1).
\end{equation}
As in the previous proofs, we use that $\Vert \mathfrak{C}(p_1,p_2)f \Vert_{L^2} \leq \Vert \mathcal{D}_x G \Vert_{L^2(\R^2)} $, where
$$
G(x,\xi) = \int_{\R^2} e^{2\pi i \sigma(x,\xi,y,\eta)} \gamma(x-y) \mathcal{D}_\xi Q_s(x,\xi,\eta)  \partial_y \mathcal{D}_y^{s-1} p_2(y,\eta) \widehat{f}(\eta)  d\eta dy.
$$
By the Cauchy-Schwartz inequality,  \eqref{e:good_s_estimate} and \eqref{e:mean_value_with_s}, we get
\begin{align*}
\vert \mathcal{D}_x G(x,\xi) \vert^2 \lesssim \big( d^{-(s-3/2)} \big)^2 M_1(\partial_\xi p_1)^2 \int_{\R} \left \vert  \int_\R e^{2\pi i( \eta - \xi)y} \mathcal{D}_x \gamma(x-y)  \partial_y \mathcal{D}^{s-1}_y p_2(y,\eta) \widehat{f}(\eta) dy \right \vert^2 d\eta.
\end{align*}
By Plancherel and Fubini, as in the end of the proof of Lemma \ref{l:castro}, we conclude that
$$
\Vert \mathcal{D}_x G \Vert_{L^2(\R^2)} \lesssim d^{-(s-3/2)}M_1( \partial_\xi p_1) N_{s}(p_2) \Vert f \Vert_{L^2}.
$$
\end{proof}

\begin{corol}
\label{c:localized_mixed_disjoint}
Let $s > 3/2$. Let $p_1 \in \mathcal{M}_{1,1}^{m_1}$ and $p_2 \in \mathcal{N}_{s,0}^{m_2}$ with $m_1+ m_2 - 1 \leq 0$. Set $\mu = \max \{m_1,m_2,0 \}$. Given $\varphi \in \mathcal{C}_c^\infty(\R)$, assume that
$$
d = \operatorname{dist} \big( 0, \supp \varphi) > 0.
$$
Then
$$
\Vert \mathfrak{C}_{t,\varphi}(p_1,p_2) \Vert_{\mathcal{L}(L^2)} \lesssim t^{s-1/2} \, d^{-(s-3/2)} \vert I_\varphi \vert^{\mu} M_{1,0}^{m_1-1}( \partial_\xi p_1 ) N_{s,0}^{m_2}( p_2 ).
$$
\end{corol}

\section{Semiclassical estimates on Sobolev spaces}

In this Section we establish some semiclassical estimates concerning the action of our operators $d(x,tD)$, $p(x,tD)$ and $p^{-1}(x,tD)$, as well as certain commutators between them, on the Sobolev spaces $H^s_t(\R)$. Despite we focus on these particular operators, we will only use their properties as operators having symbols in the classes introduced in Section \ref{s:symbol_classes}, so the techniques below can be used elsewhere.

Some techniques of paradifferential calculus are useful to extend the results of \cite{Lannes06} to the semiclassical framework. In particular, we show in Proposition \ref{p:lannes_and_texier} below how to use the techniques of \cite{Lannes06} and \cite{Tex07} to improve our Corollary \ref{c:first_corol} to fractional orders. However, the estimates of \cite{Lannes06} and \cite{Tex07} concerning commutators require more regularity in $\xi$ at the origin, in order to get satisfactory semiclassical estimates. Notice that our symbols have only one derivative in the $\xi$ variable bounded in $L^\infty$ near the origin. We avoid  the use of paradifferential calculus in our commutator estimates by requiring a bit more of regularity in the $x$ variable (see \textbf{(C2')} in Definition \ref{d:admissible}).

We first show the following lemma that link the seminorms of semiclassical symbols with those of non-semiclassical ones.

\begin{lemma}
\label{l:from_non_semiclassical_to_semiclassical}
Let $m \geq 0$ and let $a \in \mathcal{M}_{j,k}^m$. For every $t \in (0,1]$, set  $a_t(x,\xi) := a(x,t\xi)$. Then
$$
\sup_{t \in (0,1]}M_{j,k}^m(a_t) \leq M_{j,k}^m(a).
$$
Analogously, if $a \in \mathcal{N}_{s,k}^m$, then
$$
\sup_{t \in (0,1]} N_{s,k}^m(a_t) \leq N_{s,k}^m(a).
$$
\end{lemma}
\begin{proof}
For every $\alpha \leq j$ and $\beta \leq k$, one has
$$
\vert \partial_x^\alpha \partial_\xi^\beta a(x,\xi) \vert \leq M_{j,k}^m(a)(1+ \vert \xi \vert)^{m- \beta}.
$$
Then
\begin{align*}
(1 + \vert \xi \vert)^{\beta - m} \vert \partial_x^\alpha \partial_\xi^\beta a_t(x,\xi) \vert & \leq  t^\beta (1 + \vert \xi \vert)^{\beta - m} \vert \partial_x^\alpha \partial_\xi^\beta a(x,t\xi) \vert \\[0.2cm]
 & \leq M_{j,k}^m(a) t^\beta \left( \frac{1 + \vert \xi \vert}{1 + t \vert  \xi \vert} \right)^{\beta - m}.
\end{align*}
If $m - \beta \geq0$, the latter expression is uniformly bounded by $M_{j,k}^m(a)$ for all $t \in (0,1]$. Otherwise, if $\beta - m > 0$, then the function
$$
u(t) = t^\beta \left( \frac{1 + \vert \xi \vert}{1 + t \vert  \xi \vert} \right)^{\beta - m}
$$
satisfies that $u(0) = 0$, $u(1) = 1$ and $u'(t) \geq 0$ provided that $m \geq 0$. This proves the first assertion. The second follows in the same way.
\end{proof}

\begin{prop}
\label{p:firs_prop} Assume \textnormal{\textbf{(C2)}}. Let $m = \lceil m_1 \rceil$. Then the operator $p(x,tD) :  L^2(\R) \longrightarrow H_t^{-m}(\R) $ is conti\-nuous for every $t \in (0,1]$, and
\begin{align}
\label{e:Lannes}
\Vert p(x,tD) f \Vert_{H_t^{-m}(\R)} \lesssim M_{m +1,1}^{m}(p)  \Vert f \Vert_{L^2(\R)}.
\end{align}
\end{prop}

\begin{proof}Let us denote $p_t(x,\xi) = p(x,t\xi)$. Using integration by parts, for any $f \in L^2(\R)$, we have:
$$
p(x,tD)f(x) = A_Q(x;t) = \int_\R e^{2 \pi i x \xi} Q(x,\xi;t) d\xi,
$$
where
$$
Q(x,\xi;t) =  \mathcal{D}_\xi p_t(x,\xi) \Gamma^{+}_{f,\gamma}(x,\xi),
$$
and $\gamma$ is given by \eqref{e:special_function}. We estimate $\Vert A_Q \Vert_{H_t^{-m}}$ by duality, as in the proof of Lemma \ref{l:take_2d}. Take $g \in H^{m}_t(\R)$ and write
\begin{align*}
\int_\R A_Q(x;t) g(x) dx & = \int_{\R^2} e^{2 \pi i x \xi} Q(x,\xi;t) g(x) d\xi dx \\[0.2cm]
 & = \int_{\R^3} e^{2 \pi i (\lambda + \xi) x} Q(x,\xi;t) \widehat{g}(\lambda) d\lambda d\xi dx \\[0.2cm]
 & = \int_{\R^3} e^{2 \pi i (\lambda + \xi) x} Q^\dagger(x,\xi;t) (1 + 2\pi i t \xi)^{m}\widehat{g}(\lambda) d\lambda d\xi dx,
\end{align*}
where $Q^\dagger(x,\xi;t) = Q(x,\xi;t)(1 + 2\pi it\xi)^{-m}$. We next integrate by parts in $x$ to get
\begin{align}
\label{e:dual_int_by_parts}
\int_\R A_Q(x;t) g(x) dx = \int_{\R^2} e^{2 \pi i \xi x}\mathcal{D}_x^{m+1} Q^\dagger(x,\xi;t) \, \Gamma_{g}(x,\xi;t) d\xi dx,
\end{align}
where
\begin{align*}
\Gamma_g(x,\xi;t) & = \int_\R \frac{ (1 + 2\pi it \xi)^{m} \widehat{g}(\lambda) e^{2\pi i \lambda x}}{(1 + 2\pi i (\lambda + \xi))^{m+1}} d\lambda  \\[0.2cm]
 & = \sum_{j=0}^{m} (-1)^j {m \choose j} \int_\R \frac{(1+ 2\pi it(\lambda + \xi))^{m - j} (2\pi i t \lambda)^j \widehat{g}(\lambda) e^{2\pi i \lambda x}}{(1 + 2\pi i (\lambda + \xi))^{m + 1} } d\lambda.
\end{align*}
Using Cauchy-Schwartz inequality and Lemma \ref{l:elementary} we obtain that
\begin{align*}
\left \vert \int_\R A_Q(x) g(x) dx \right \vert & \lesssim \Vert \mathcal{D}^{m+1}_x Q^\dagger \Vert_{L^2(\R^2)} \Vert g \Vert_{H^{m}_t(\R)} \\[0.2cm]
 & \lesssim M_{m+1,1}^0 \big( p^\dagger_t  \big) \Vert f \Vert_{L^2(\R)} \Vert g \Vert_{H^{m}_t(\R)},
\end{align*}
where $p_t^\dagger(x,\xi) = p^\dagger(x,t\xi)$ and $p^\dagger(x,\xi) = p(x,\xi)(1+2\pi i \xi)^{-m}$.
Finally, to obtain \eqref{e:Lannes}, it is sufficient to use Lemma \ref{l:from_non_semiclassical_to_semiclassical}.
\end{proof}

\begin{corol}
\label{c:first_corol} Assume \textnormal{\textbf{(C2')}} and set $m = \lceil m_1 \rceil$. Then
\begin{equation}
\label{e:sobolev_estimate}
\Vert p(x,tD) f \Vert_{H_t^{-m}(\R)} \lesssim  \Vert v \Vert_{H^{s}\times H^{s}} \Vert f \Vert_{L^2(\R)}.
\end{equation}
\end{corol}

\begin{proof}
We write $p(x,\xi) = \Pi_p(x,\xi) + \Sigma_p(0,\xi)$. Since $\Sigma_p(0,\xi) \in \mathcal{M}_{\infty,1}^{m}$ is a Fourier multiplier, one has
$$
\Vert \Sigma_p(0,tD) f \Vert_{H^{-m}_t} \lesssim \Vert f \Vert_{L^2}.
$$
It remains to show that
$$
\Vert \Pi_p(x,tD) f \Vert_{H^{-m}_t} \lesssim  \Vert v \Vert_{H^{s}\times H^{s}} \Vert f \Vert_{L^2}.
$$
To do this, we write
$$
\Pi_p(x,tD)f(x) = A_K(x;t) = \int_\R e^{2 \pi i x \xi} K(x,\xi;t) d\xi,
$$
where $K(x,\xi;t) = \Pi_p(x,t\xi) \widehat{f}(\xi)$. We estimate $\Vert A_K \Vert_{H_t^{-m}}$ by duality, exactly as in the proof of Proposition \ref{p:firs_prop}, with $K$ instead of $Q$. Take $g \in H^{m}_t(\R)$. Next integrate by parts in $x$ to get
\begin{align}
\label{e:dual_int_by_parts2}
\int_\R A_K(x;t) g(x) dx = \int_{\R^2} e^{2 \pi i \xi x}\mathcal{D}_x^{m+1} K^\dagger(x,\xi;t) \, \Gamma_{g}(x,\xi;t) d\xi dx,
\end{align}
where $K^\dagger(x,\xi;t) = K(x,\xi;t)(1 + 2\pi it\xi)^{-m}$ and
\begin{align*}
\Gamma_g(x,\xi;t) & = \int_\R \frac{ (1 + 2\pi it \xi)^{m} \widehat{g}(\lambda) e^{2\pi i \lambda x}}{(1 + 2\pi i (\lambda + \xi))^{m+1}} d\lambda  \\[0.2cm]
 & = \sum_{j=0}^{m} (-1)^j {m \choose j} \int_\R \frac{(1+ 2\pi it(\lambda + \xi))^{m - j} (2\pi i t \lambda)^j \widehat{g}(\lambda) e^{2\pi i \lambda x}}{(1 + 2\pi i (\lambda + \xi))^{m + 1} } d\lambda.
\end{align*}
Using Cauchy-Schwartz inequality and Lemma \ref{l:elementary} we obtain that
\begin{align*}
\left \vert \int_\R A_K(x) g(x) dx \right \vert & \lesssim \Vert \mathcal{D}^{m+1}_x K^\dagger \Vert_{L^2(\R^2)} \Vert g \Vert_{H^{m}_t(\R)} \\[0.2cm]
 & \lesssim N_{m+1}( \Pi^\dagger_{p,t}) \Vert f \Vert_{L^2(\R)} \Vert g \Vert_{H^{m}_t(\R)},
\end{align*}
where $\Pi^\dagger_{p,t}(x,\xi) = \Pi^\dagger_p(x,t\xi)$ and $\Pi^\dagger_{p}(x,\xi) = \Pi_p(x,\xi)(1 + 2\pi i \xi)^{-m}$. Finally, making use of Lemma \ref{l:from_non_semiclassical_to_semiclassical} and the fact that $\Pi_p$ is smooth in the image of $v$, hence Moser's inequality applies, we conclude that
\begin{align*}
N_{m+1}( \Pi^\dagger_{p,t})  \lesssim N_{m+1,0}^{m}(\Pi_p) \leq C(\Vert v \Vert_{L^\infty})(1 + \Vert v \Vert_{H^s \times H^s}),
\end{align*}
provided that $s \geq 2$.
\end{proof}

Notice that the proofs of Proposition \ref{p:firs_prop} and Corollary \ref{c:first_corol} are particularly simple due to the use of the Leibniz rule as after \eqref{e:dual_int_by_parts2}. The use of paradifferential calculus as in \cite{Lannes06} and \cite{Tex07} allows us to improve this result to the case $m_1$ being non-integer.

\begin{prop}
\label{p:lannes_and_texier}
Assume \textnormal{\textbf{(C2')}}. Then
\begin{equation}
\label{e:fractional_estimate}
\Vert p(x,tD) f \Vert_{H_t^{-m_1}(\R)} \lesssim  \Vert v \Vert_{H^{s}\times H^{s}} \Vert f \Vert_{L^2(\R)},
\end{equation}
provided that $s > m_1$.
\end{prop}

\begin{proof}
\label{r:to_Lannes}
We write $p(x,\xi) = \Pi_p(x,\xi) + \Sigma_p(0,\xi)$. Since $\Sigma_p(0,tD)$ is a Fourier multiplier with symbol belonging to $\mathcal{M}^{m_1}_{\infty,1}$, the it satisfies \eqref{e:fractional_estimate} trivially. It is then sufficient to prove the result for  $\Pi_p(x,tD)$.

To this aim, set  $\sigma(x,\xi) := \Pi_p(x,\xi)$ and consider the decomposition of \cite{Lannes06}:
$$
\sigma = \sigma_{lf} + \sigma_I + \sigma_{II} + \sigma_R.
$$
We will estimate each of these terms separately. First, for the \textit{low-frequency} term $\sigma_{lf}$ we just observe that $\sigma_{lf} \in \mathcal{N}_{s,1}^m$ for every $m \in \R$. Then we can use\footnote{Ona can also mimic the proof of Corollary \ref{c:first_corol} with $m= 0$ instead of $m_1$.} \cite[Corollary 2.2]{Hwang87} together with Lemma \ref{l:from_non_semiclassical_to_semiclassical} and Moser's inequality to get
$$
\Vert \sigma_{lf}(x,tD) f \Vert_{H_t^{-m_1}(\R)} \leq \Vert \sigma_{lf}(x,tD) f \Vert_{L^2(\R)} \lesssim N_{1}(\sigma_{lf}) \Vert f \Vert_{L^2(\R)} \lesssim \Vert v \Vert_{H^{s}\times H^{s}} \Vert f \Vert_{L^2(\R)}.
$$
The second term $\sigma_I$ is smooth in both variables, so one can estimate $\Vert \sigma_I(x,tD) \Vert_{H^{-m_1}_t}$ by the right hand side of \eqref{e:fractional_estimate} just using classical tools. We refer for instance to \cite[Prop. 23]{Tex07}.

In order to bound the terms $\sigma_{II}$ and $\sigma_R$, we adapt the proof of \cite[Prop. 25, (ii)]{Lannes06} to our context. For the term $\sigma_{II}$, we proceed as follows. Using Remark \ref{r:isometry}, we have that
$$
\Vert \sigma_{II}(x,td) f \Vert_{H^{-m_1}_t} = \Vert \sigma_{II}^t(x,D) U_t^* f \Vert_{H^{-m_1}}.
$$
Moreover, by the second estimate of \cite[Prop. 20]{Lannes06}, we have
$$
\Vert \sigma_{II}^t(x,D) U_t^* f \Vert_{H^{-m_1}} \lesssim N_{s-k,2}^{m_1}(\nabla_x^k \sigma^t) \Vert \Vert  U_t^* f \Vert_{L^2} = N_{s-k,2}^{m_1}(\nabla_x^k \sigma^t) \Vert \Vert f \Vert_{L^2},
$$
for every $k < s$. Using that
\begin{equation}
\label{e:texier}
N_{s-k,2}^{m_1}(\nabla_x^k \sigma^t) \leq t^{k-1/2} N_{s,2}^{m_1}(\sigma),
\end{equation}
and Moser's inequality, we obtain the desired estimate.

Finally, the term $\sigma_R$ can be bounded in a similar way. Using again Remark \ref{r:isometry} we have
$$
\Vert \sigma_{R}(x,tD) f \Vert_{H^{-m_1}_t} = \Vert \sigma_{R}^t(x,D) U_t^* f \Vert_{H^{-m_1}}.
$$
Moreover, using the first estimate of \cite[Prop. 23]{Lannes06} we get
$$
\Vert \sigma_{R}^t(x,D) U_t^* f \Vert_{H^{-m_1}} \lesssim N_{s-k,2}^{m_1}(\nabla_x^k \sigma^t) \Vert \Vert  U_t^* f \Vert_{L^2} = N_{s-k,2}^{m_1}(\nabla_x^k \sigma^t)  \Vert f \Vert_{L^2},
$$
for every $k < s$. By a further use of \eqref{e:texier} and Moser's inequality, we conclude.
\end{proof}

We next deal with semiclassical commutator estimates.
\begin{prop}
\label{p:third_prop} Assume \textnormal{\textbf{(C2)}}. Then there exists $C_2 > 0$ such that, for every $t \in (0,1]$:
\begin{equation}
\label{e:easy}
 \Vert t^{-1} \mathfrak{C}_t(p^{-1}, d) \Vert_{\mathcal{L}(H^{-m_1}_t; L^2)} \leq C_2.
\end{equation}
\end{prop}

\begin{proof}
Let $f \in H^{-m_1}_t(\R)$, define  $g \in L^2(\R)$ by $\widehat{f}(\xi) = (1 +2\pi it \xi)^{m_1} \widehat{g}(\xi)$.
We have
\begin{align*}
t^{-1} \mathfrak{C}_t(p^{-1},d) f(x) & = t^{-1} \mathfrak{C}_t(p^{-1},d^\dagger) g(x),
\end{align*}
where $d^\dagger(x,\xi) = d(x,\xi) (1 + 2 \pi i\xi)^{m_1}$.
We observe that $d^\dagger \in \mathcal{M}^{m_1}_{N,1}$, while $p^{-1} \in \mathcal{M}^{-m_2}_{N,1}$. Then, to prove \eqref{e:easy}, it is sufficient to show that
\begin{equation}
\label{e:from_L^2_to_L^2}
\Vert t^{-1} \mathfrak{C}_t(p^{-1},d^\dagger) g \Vert_{L^2(\R)} \leq C_2 \Vert g \Vert_{L^2}.
\end{equation}
To this aim, we consider a partition of unity as follows: Let $\varphi, \psi \in \mathbb{C}_c^\infty(\R)$ so that
\begin{align*}
\supp \varphi & \subset \left \{ \frac{1}{2} \leq \vert \xi \vert \leq 2 \right \}, \quad \supp \psi  \subset \left \{ \vert \xi \vert < 1 \right \},
\end{align*}
and such that, setting
\begin{equation}
\label{e:partition_unity}
\left \lbrace \begin{array}{l}
\varphi_{-1}(\xi)  := \psi(\xi) \\[0.2cm]
\varphi_j(\xi)  := \displaystyle \varphi \left( \frac{\xi}{2^j} \right), \quad j \geq 0,
\end{array} \right.
\end{equation}
one has: $1 \equiv  \sum_{j=-1}^\infty \varphi_j(\xi)$.
We then write
\begin{align*}
\mathfrak{C}_t(p^{-1},d^\dagger) & = \sum_{j=-1}^\infty \mathfrak{C}_{t,\varphi_j}(p^{-1}, d^\dagger),
\end{align*}
where the terms $\mathfrak{C}_{t,\varphi_j}(p^{-1}, d^\dagger)$ are defined by \eqref{e:localized_semiclassical_commutator}. By Corollary \ref{c:localized_castro}, we have
\begin{align*}
\Vert \mathfrak{C}_{t,\varphi_{-1}}(p^{-1},d^\dagger) \Vert_{\mathcal{L}(L^2)}  \lesssim t \, \mathfrak{M}(p^{-1},d^\dagger).
\end{align*}
For $j \geq 0$, we use Corollary \ref{c:localized_disjoint_supports}, together with condition $N-3/2 > m_1$, to obtain
\begin{align*}
\Vert \mathfrak{C}_{t,\varphi_{j}}(p^{-1},d^\dagger) \Vert_{\mathcal{L}(L^2)} &  \lesssim t^{N-1} 2^{-j(N-3/2 -m_1)} \, \mathfrak{M}_N(p^{-1},d^\dagger).
\end{align*}
Summing in $j$, we obtain the claim provided that $N-3/2 > m_1$.
\end{proof}

Proposition \ref{p:third_prop} allows us to improve Proposition \ref{p:firs_prop} in the following way:
\begin{corol}
\label{c:better_than_prop_2}
Assume \textnormal{\textbf{(C2)}}. Then the operator $p(x,tD) :  L^2(\R) \longrightarrow H_t^{-m_1}(\R) $ is conti\-nuous for every $t \in (0,1]$, and
\begin{align}
\label{e:Lannes_2}
\Vert p(x,tD) f \Vert_{H_t^{-m_1}(\R)} \lesssim  \Vert f \Vert_{L^2(\R)}.
\end{align}
\end{corol}
\begin{proof}
Denoting $\langle \xi \rangle = 1 + 2\pi i \xi$, and $\langle tD \rangle$ its semiclassical quantization, we have:
\begin{align*}
\Vert p(x,tD) f \Vert_{H^{-m_1}_t} & = \Vert \langle tD \rangle^{-m_1} p(x,tD) f \Vert_{L^2} \\[0.2cm]
 & = \Vert p \langle \xi \rangle^{-m_1} (x,tD)f + \mathfrak{C}_t(p, \langle \xi \rangle^{-m_1})f \Vert_{L^2} \\[0.2cm]
 & \leq \Vert p \langle \xi \rangle^{-m_1} (x,tD)f \Vert_{L^2}  + \Vert \mathfrak{C}_t(p, \langle \xi \rangle^{-m_1})f \Vert_{L^2}.
\end{align*}
The first term is bounded by Calderón-Vaillancourt Theorem \cite[Thm. 2]{Hwang87}. The second one is also bounded by \eqref{e:from_L^2_to_L^2}, after replacing $p^{-1}$ by $p$ and $d^\dagger$ by $\langle \xi \rangle^{-m_1}$.
\end{proof}

\begin{corol}
\label{c:last_corol} Assume \textnormal{\textbf{(C2')}}. Then there exists $C_2 > 0$ such that, for every $t \in (0,1]$:
\begin{equation}
\label{e:easy_2}
 \Vert t^{-1} \mathfrak{C}_t(p^{-1}, d) \Vert_{\mathcal{L}(H^{-m_1}_t; L^2)} \leq C_2.
\end{equation}
\end{corol}

\begin{proof}
We reason as in the proof of Proposition \ref{p:third_prop}. Let $f \in H^{-m_1}_t(\R)$, define  $g \in L^2(\R)$ by $\widehat{f}(\xi) = (1 +2\pi it \xi)^{m_1} \widehat{g}(\xi)$.
We have again
\begin{align*}
t^{-1} \mathfrak{C}_t(p^{-1},d) f(x) & = t^{-1} \mathfrak{C}_t(p^{-1},d^\dagger) g(x),
\end{align*}
where
$$
 d^\dagger(x,\xi) = d(x,\xi) (1 +2 \pi i\xi)^{m_1}.
$$
In this case, we have $\Pi_{d^\dagger} \in \mathcal{N}_{s,1}^{m_1}$ and $\Pi_{p^{-1}} \in \mathcal{N}_{s,1}^{-m_2}$. We aim at proving that
\begin{equation}
\label{e:from_L^2_to_L^2_second}
\Vert t^{-1} \mathfrak{C}_t(p^{-1},d^\dagger) g \Vert_{L^2(\R)} \leq C_2 \Vert g \Vert_{L^2}.
\end{equation}
To this aim, we will localize the commutators with the partition of unity \eqref{e:partition_unity}, and we will use Corollaries \ref{c:localized_commutator_in_sobolev}, \ref{c:localized_mixed_commutator}, \ref{c:localized_disjoint_in_sobolev} and \ref{c:localized_mixed_disjoint}, instead of Corollaries \ref{c:localized_castro} and \ref{c:localized_disjoint_supports}. Let us write
\begin{align*}
d^\dagger(x,\xi) & = \Pi_{d^\dagger}(x,\xi) + \Sigma_{d^\dagger}(0,\xi), \\[0.2cm]
p^{-1}(x,\xi) & = \Pi_{p^{-1}}(x,\xi) + \Sigma_{p^{-1}}(0,\xi).
\end{align*}
Hence
\begin{align*}
\mathfrak{C}_t(p^{-1},d^\dagger) & \\[0.2cm]
 & \hspace*{-1cm}  = \mathfrak{C}_t(\Pi_{p^{-1}},\Pi_{d^\dagger}) + \mathfrak{C}_t(\Sigma_{p^{-1}}(0,\cdot),\Pi_{d^\dagger}) + \mathfrak{C}_t(\Pi_{p^{-1}},\Sigma_{d^\dagger}(0,\cdot)) + \mathfrak{C}_t( \Sigma_{p^{-1}}(0,\cdot),\Sigma_{d^\dagger}(0,\cdot)) \\[0.2cm]
 & \hspace*{-1cm} = A_1 + A_2 + A_3 + A_4.
\end{align*}
First, we have that
\begin{align*}
A_3 & = \mathfrak{C}_t(\Pi_{p^{-1}},\Sigma_{d^\dagger}(0,\cdot)) = 0, \\[0.2cm]
A_4 & = \mathfrak{C}_t( \Sigma_{p^{-1}}(0,\cdot),\Sigma_{d^\dagger}(0,\cdot)) = 0.
\end{align*}
To estimate $A_1$, instead of Corollaries \ref{c:localized_castro} and \ref{c:localized_disjoint_supports}, we use Corollaries \ref{c:localized_commutator_in_sobolev} and \ref{c:localized_disjoint_in_sobolev}. Notice that we can replace the decayment $2^{-j(N-3/2)}$ by $2^{-j(s-3/2)}$ associated with the distance between $0$ and the support of $\varphi_j$. The condition $s - 3/2 > m_1$ suffices then to obtain the claim.

Finally, to deal with $A_2$, we use Corollary \ref{c:localized_mixed_commutator} (the first statement) instead of Corollary \ref{c:localized_castro}, and Corollary \ref{c:localized_mixed_disjoint} instead of Corollary \ref{c:localized_disjoint_supports}.
\end{proof}

\begin{prop}
\label{p:good_commutators}
Assume \textnormal{\textbf{(C2)}}. Then, for every $0 \leq m \leq \lceil m_1\rceil$:
 $$
 \Vert \mathfrak{C}_t(p,p^{-1}) \Vert_{\mathcal{L}(H_t^{-m},H_t^{-m})} \leq t C_1, \quad t \in (0,1].
 $$
\end{prop}

\begin{proof}
By interpolation, be can assume that $m$ is an integer. Let $f \in H_t^{-m}(\R)$, we define $g \in L^2(\R)$  via $\widehat{f}(\xi) = (1 + 2\pi it \xi)^{m} \widehat{g}(\xi)$.
Then
$$
\mathfrak{C}_t(p,p^{-1})f = \mathfrak{C}_t(p,q) g,
$$
where $q(x,\xi) = p^{-1}(x,\xi)(1 + 2 \pi i\xi)^{m}$ belongs to $\mathcal{M}^{m-m_2}_{N,1}$. It is then sufficient to show that
$$
\Vert \mathfrak{C}_t(p,q)g \Vert_{H_t^{-m}(\R)} \leq t C_1  \Vert g \Vert_{L^2(\R)}.
$$
Using the partition of unity \eqref{e:partition_unity}, we split the sum as
$$
\Vert \mathfrak{C}_t(p,q)g \Vert_{H_t^{-1}(\R)} \leq \sum_{j=-1}^{\infty} \Vert \mathfrak{C}_{t,\varphi_j}(p,q)g \Vert_{H_t^{-1}(\R)}.
$$
We claim that there exists $\alpha > 0$ such that, for every $h \in H_t^{m}(\R)$,
$$
\left \vert \int_\R \mathfrak{C}_{t,\varphi_j}(p,q)g(x) h(x) dx \right \vert \lesssim t \, 2^{-\alpha j} \Vert g \Vert_{L^2(\R)} \Vert h \Vert_{H_t^{m}(\R)}.
$$
To show this, we write
$$
\int_\R \mathfrak{C}_{t,\varphi_j}(p,q)g(x) h(x) dx  = \int_{\R^2} e^{2\pi i x \xi} Q_t^{j}(x,\xi) h(x) d\xi dx,
$$
where
$$
Q_t^{j}(x,\xi) = \int_{\R^2} e^{2\pi i (-\xi y + y\eta)} \varphi_j \big( t(\xi - \eta) \big)\big( p(x,t\xi) - p(x,t\eta) \big) q(y,t\eta) \widehat{g}(\eta) d\eta dy.
$$
Morover,
\begin{align*}
\int_{\R^2} e^{2\pi i x \xi} Q_t^{j}(x,\xi) h(x) d\xi dx & = \int_{\R^3} e^{2\pi i x(\xi + \lambda)} \frac{Q^{j}_t(x,\xi)}{(1 + 2\pi i t\xi)^{m}} \big( 1 + 2\pi i t\xi)^{m} \big) \widehat{h}(\lambda) d\lambda \\[0.2cm]
 &  = \sum_{k=0}^{m} {m \choose k} \int_{\R^3} \frac{Q_t^{j}(x,\xi)}{(1 + 2\pi i\xi)^{m}} (1 + 2\pi i t(\lambda + \xi))^{m - k}(2\pi i t \lambda)^k \widehat{h}(\lambda)d\lambda \\[0.2cm]
 & = \sum_{k=0}^{m} I_k.
\end{align*}
Using integration by parts in the $x$ variable, and the following identity,
$$
\frac{1}{(1 + 2\pi i (\xi + \lambda))^{m+1}} (1 + \partial_x)^{m+1} e^{2\pi i x(\xi + \lambda)} = e^{2\pi i x(\xi + \lambda)},
$$
as in the proof of Proposition \ref{p:firs_prop}, we obtain, for $0 \leq k \leq m$, that
\begin{align*}
\vert I_k \vert & \leq \left \Vert \frac{\mathcal{D}_x^{m+1} Q^{j}_t(x,\xi)}{(1 + 2\pi i t \xi)^{m}} \right \Vert_{L^2(\R^2)} \Vert h \Vert_{H^k_t(\R)}.
\end{align*}
Recalling the proofs of Corollary \ref{c:localized_castro} and Proposition \ref{p:third_prop}, the claim is obtained in a similar way. Notice that
\begin{align*}
 t \cdot \sup_{(x,\xi,\eta) \in \R^3} \left \vert \frac{\varphi_{-1}\big(t(\xi - \eta) \big)}{(1+ 2\pi i t \xi)^m} \frac{\mathcal{D}^{m+1}_x\big( p(x,t\xi) - p(x,t\eta) \big)(1+2\pi i t \eta)^{m-m_2}}{2\pi i t(\xi - \eta)}  \right \vert &  \\[0.2cm]
 & \hspace*{-8cm} =  t \cdot \sup_{(x,\xi,\eta) \in \R^3} \left \vert \frac{\varphi_{-1}(\xi - \eta)}{(1+ 2\pi i  \xi)^m} \frac{\mathcal{D}^{m+1}_x\big( p(x,\xi) - p(x,\eta) \big)(1+2\pi i  \eta)^{m-m_2}}{2\pi i (\xi - \eta)}  \right \vert \\[0.2cm]
 & \hspace*{-8cm} \leq t \, \Vert \varphi_{-1} \Vert_{L^\infty}  \sup_{x \in \R} \sup_{ (\xi, \eta, \rho) \in \Omega_{\varphi_{-1}}} \left \vert \frac{\mathcal{D}^{m+1}_x \partial_\eta p(x,\eta) (1+ 2\pi i \rho)^{m-m_2}}{(1+ 2\pi i  \xi)^m}  \right \vert  \\[0.2cm]
 & \hspace*{-8cm} \leq t \,M_{m+1,0}^{m_1 - 1}(\partial_\xi p) \sup_{(\xi,\eta,\rho) \in \Omega_{\varphi_{-1}}} (1+ \vert \eta \vert)^{m_1-1} (1+ \vert \xi \vert)^{-m} (1 + \vert \rho \vert)^{m-m_2} \\[0.2cm]
 & \hspace*{-8cm} \lesssim t.
\end{align*}
Similarly, for $j \geq 0$,
\begin{align*}
 t^{N-1} \cdot \sup_{(x,\xi,\eta) \in \R^3} \left \vert \frac{\varphi_{j}\big(t(\xi - \eta) \big)}{(1+ 2\pi i t \xi)^m} \frac{\mathcal{D}^{m+1}_x\big( p(x,t\xi) - p(x,t\eta) \big)(1+2\pi i t \eta)^{m-m_2}}{t (2\pi i (\xi - \eta))^{N-1}}  \right \vert &  \\[0.2cm]
 & \hspace*{-11cm} \lesssim  t^{N-1} 2^{-j(N-2)} \cdot \sup_{(x,\xi,\eta) \in \R^3} \left \vert \frac{\varphi_{j}(\xi - \eta)}{(1+ 2\pi i  \xi)^m} \frac{\mathcal{D}^{m+1}_x\big( p(x,\xi) - p(x,\eta) \big)(1+2\pi i  \eta)^{m-m_2}}{2\pi i (\xi - \eta)}  \right \vert \\[0.2cm]
 & \hspace*{-11cm} \lesssim t^{N-1} 2^{-j(N-2)} \sup_{x \in \R} \sup_{ (\xi, \eta,\rho) \in \Omega_{\varphi_{j}}} \left \vert \frac{\mathcal{D}^{m+1}_x \partial_\eta p(x,\eta)(1+2\pi i \rho)^{m-m_2}}{(1+ 2\pi i t \xi)^m}  \right \vert  \\[0.2cm]
 & \hspace*{-11cm} \lesssim t^{N-1} 2^{-j(N-2)} \,M_{m+1,0}^{m_1 - 1}(\partial_\xi p) \sup_{(\xi,\eta,\rho) \in \Omega_{\varphi_{j}}} (1+ \vert \eta \vert)^{m_1-1} (1+ \vert \xi \vert)^{-m} (1+\vert \rho \vert)^{m-m_2} \\[0.2cm]
 & \hspace*{-11cm} \lesssim t^{N-1} 2^{-j(N-2)} 2^{jm_1}.
\end{align*}
Thus, using \eqref{e:extra_decayment}, that $p \in \mathcal{M}^{m}_{N,1}$ and $q \in \mathcal{M}^{m-m_2}_{N,1}$, condition $N-3/2-m_1 > 0$, and the ideas above, is sufficient to finish the proof.
\end{proof}
Assembling the previous ideas together with those of the proof of Corollary \ref{c:last_corol}, we get also the following:
\begin{corol}
\label{c:now_with_sobolev}
Assume  \textnormal{\textbf{(C2')}}. Then, for every $0 \leq m \leq \lceil m_1 \rceil$,
 \begin{equation}
 \label{e:remainder_for_invertibility}
 \Vert \mathfrak{C}_t(p,p^{-1}) \Vert_{\mathcal{L}(H_t^{-m},H_t^{-m})} \leq t C_1, \quad t \in (0,1].
 \end{equation}
\end{corol}

We will also require the following coercivity property for $p^{-1}(x,tD)$:

\begin{prop}
\label{l:coercividad}
Let $\langle tD \rangle$ be the semiclassical quantization of the symbol $\langle \xi \rangle = 1 +2\pi  i \xi$. Then there exists $C > 0$ and $0 < T \leq 1$ such that
$$
\Vert f \Vert_{H^{-m_1}_t(\R)} = \Vert \langle tD \rangle^{-m_1} f \Vert_{L^2(\R)} \leq C \Vert p^{-1}(x,tD) f \Vert_{L^2(\R)}, \quad f \in \mathcal{D}'(\R),
$$
for every $0 < t \leq T$.
\end{prop}

\begin{proof}
We formally have
$$
\langle tD \rangle^{-m_1} f(x) = \langle tD \rangle^{-m_1} \big(p(x,tD) p^{-1}(x,tD) \big)^{-1} p(x,tD) p^{-1}(x,tD) f(x).
$$
By Corollary \ref{c:better_than_prop_2} with hypothesis \textbf{(C2)} (resp. Proposition \ref{p:lannes_and_texier} with \textbf{(C2')}), $p(x,tD) : L^2(\R) \to H_t^{-m_1}(\R)$ is continuous, uniformly in $t \in (0,1]$. Then it is sufficient to show that the operator $p(x,tD) p^{-1}(x,tD)$ is invertible in $H_t^{-m_1}(\R)$ with continuous inverse, uniformly in $t \in (0,T]$  for some $T > 0$. To this aim, notice that
$$
p(x,tD) p^{-1}(x,tD) = I + \mathfrak{C}_t(p,p^{-1}).
$$
Using Proposition \ref{p:good_commutators} together with hypothesis \textnormal{\textbf{(C2)}} (resp. Corollary \ref{c:now_with_sobolev} together with \textnormal{\textbf{(C2')}}), we have that
$$
\Vert \mathfrak{C}_t(p,p^{-1}) \Vert_{\mathcal{L}(H^{-m_1}_t)} \leq tC_1.
$$
Then, there exists $T> 0$ sufficiently small such that $\Vert \mathfrak{C}_t(p,p^{-1}) \Vert_{\mathcal{L}(H^{-m_1}_t)} < 1$ for $t \in (0,T]$. Then, the operator $p(x,tD) p^{-1}(x,tD)$ is invertible and has continuous inverse in $H_t^{-m_1}(\R)$, uniformly bounded for $t \in (0,T]$.
\end{proof}

\section{Proof of Theorems \ref{t:main_theorem1} and \ref{t:localization}}

\begin{proof}[Proof of Theorem \ref{t:main_theorem1}]
First observe that
$$
\frac{d}{dt} p^{-1}(x,t\xi) = -t^{-1} d(x,t\xi) p^{-1}(x,t\xi).
$$
In view of (IVP) and \eqref{e:semicommutator}, we have
\begin{align*}
\frac{d}{dt} \Vert p^{-1}(x,tD) f(t) \Vert_{L^2}^2 & = \frac{d}{dt} \big \langle p^{-1}(x,tD) f(t), p^{-1}(x,tD) f(t) \big \rangle_{L^2} \\[0.2cm]
 & = 2 \operatorname{Re} \big \langle t^{-1} \mathfrak{C}_t(p^{-1},d) f(t), p^{-1}(x,tD) f(t) \big \rangle_{L^2}.
\end{align*}
Finally, using Proposition \ref{p:third_prop} and Lemma \ref{l:coercividad}, we conclude that
\begin{align*}
\Vert  t^{-1} \mathfrak{C}_t(p^{-1},d) f(t) \Vert_{L^2(\R)} & = \Vert  t^{-1} \mathfrak{C}_t(p^{-1},d) \langle tD \rangle^{m_1} \langle tD \rangle^{-m_1} f(t) \Vert_{L^2(\R)} \\[0.2cm]
 & \leq \Vert  t^{-1} \mathfrak{C}_t(p^{-1},d) \Vert_{\mathcal{L}(H_t^{-m_1}, L^2)} \Vert \langle tD \rangle^{m_1} \Vert_{\mathcal{L}(L^2,H_t^{-m_1})} \Vert \langle tD \rangle^{-m_1} f \Vert_{L^2(\R)} \\[0.2cm]
 & \leq C_T \Vert p^{-1}(x,tD) f(t) \Vert_{L^2(\R)}.
\end{align*}

%On the other hand,
%\begin{align*}
%t^{-1} \mathfrak{C}_t(p^{-1},d) = t^{-1} \mathfrak{C}_t(p^{-1},d)p(x,tD)p^{-1}(x,tD) - t^{-1} \mathfrak{C}_t(p^{-1},d) \mathfrak{C}_t(p,p^{-1}).
%\end{align*}

%Using Propositions \ref{p:firs_prop}, \ref{p:good_commutators} and \ref{p:third_prop}, with assumption \textbf{(C2)} (resp. Corollaries \ref{c:first_corol}, \ref{c:now_with_sobolev} and \ref{c:last_corol}, with assumption \textbf{(C2')}), and identity \eqref{e:approximate_inverse},  we infer the existence of constants $C_1, C_2, C_3 > 0$ such that
%\begin{align*}
%\Vert t^{-1} \mathfrak{C}_t(p^{-1},d)  \Vert_{\mathcal{L}(L^2)} & \leq \Vert t^{-1} \mathfrak{C}_t(p^{-1},d) \Vert_{\mathcal{L}(H^{-1}_t;L^2)} \Vert p(x,tD) \Vert_{\mathcal{L}(L^2;H^{-1}_t)} \Vert p^{-1}(x,tD) \Vert_{\mathcal{L}(L^2)} \\[0.2cm]
%& \quad  + \Vert t^{-1} \mathfrak{C}_t(p^{-1},d) \Vert_{\mathcal{L}(L^2)} \Vert \mathfrak{C}_t(p,p^{-1}) \Vert_{\mathcal{L}(L^2)} \\[0.2cm]
%& \leq  C_2 C_3 \, \Vert p^{-1}(x,tD) \Vert_{\mathcal{L}(L^2)} + tC_1 \Vert t^{-1} \mathfrak{C}_t(p^{-1},d) \Vert_{\mathcal{L}(L^2)}.
%\end{align*}
%Hence, for $t < 1/C_1$, we obtain
%$$
%\Vert t^{-1} \mathfrak{C}_t(p^{-1},d) f(t) \Vert_{L^2} \leq \frac{C_2 C_3}{1- tC_1} \Vert p^{-1}(x,tD)f(t) \Vert_{L^2}.
%$$
%Taking $T < 1/C_1$, we have shown that, for every $t \in (0,T]$,
%$$
%\frac{d}{dt} \Vert p^{-1}(x,tD) f(t) \Vert_{L^2}^2 \leq \frac{2  C_2 C_3}{1 - TC_1} \Vert p^{-1}(x,tD)f(t) \Vert_{L^2}^2.
%$$
\end{proof}

\begin{proof}[Proof of Theorem \ref{t:localization}]
We first prove the inequality on the left. To this aim, we define a new function $c_\varepsilon : \R \to \R$ satisfying \textbf{(C2')}, such that $c_\varepsilon(x) = c(x)$ in $I_\varepsilon$, and such that $c_\varepsilon^{-} \leq c_\varepsilon(x) \leq c_\varepsilon^+$ for every $x \in \R$. We then define a new symbol $p_\varepsilon^{-1}$ by
$$
p^{-1}_\varepsilon(x,\xi) = (1 + c_\varepsilon(x) \vert \xi \vert)^{-1/c_\varepsilon(x)}.
$$
Notice that $\chi_\varepsilon(x)p^{-1}(x,tD)f(x) = \chi_\varepsilon(x)p^{-1}_\varepsilon(x,tD)f(x)$ and
\begin{align}
\label{e:localize}
\chi_\varepsilon(x)p^{-1}_\varepsilon(x,tD)f(x) = p^{-1}_\varepsilon(x,tD) (\chi_\varepsilon f)(x) - \mathfrak{C}_t(p_\varepsilon^{-1}, \chi_\varepsilon)f(x).
\end{align}
By Corollary \ref{c:last_corol} (notice that we can replace $p^{-1}$ by $p_\varepsilon^{-1}$ and $d$ by $\chi_\varepsilon$ in the statement),
\begin{equation}
\label{e:commutator_with_bump}
\Vert \mathfrak{C}_t(p_\varepsilon^{-1}, \chi_\varepsilon)f \Vert_{L^2(\R)} \leq C_0 t\Vert f \Vert_{H^{-m_1}_t(\R)}.
\end{equation}
On the other hand, notice that
$$
\langle tD \rangle^{s_-} p_\varepsilon(x,tD) = p_\varepsilon \langle \xi \rangle^{s_-}(x,tD) + \mathfrak{C}_t(\langle \xi \rangle^{s_-}, p_\varepsilon),
$$
where $\langle \xi \rangle = 1 + 2 \pi i \xi$. Hence, using that $\langle \xi \rangle^{s_-} \in \mathcal{M}^{s_-}_{\infty,\infty}$, and the commutator estimate \eqref{e:from_L^2_to_L^2_second} with $p_\varepsilon$ replacing $d^\dagger$ and $\langle \xi \rangle^{s_-}$ replacing $p^{-1}$ yields that
\begin{align*}
\Vert p_\varepsilon(x,tD) f \Vert_{H^{s_-}} & = \Vert \langle tD \rangle^{s_-} p_\varepsilon(x,tD) \Vert_{L^2} \\[0.2cm]
 &  \leq \Vert p_\varepsilon \langle \xi \rangle^{s_-}(x,tD) f \Vert_{L^2} + \Vert \mathfrak{C}_t(\langle \xi \rangle^{s_-}, p_\varepsilon) \Vert_{L^2} \\[0.2cm]
 & \leq \Vert p_\varepsilon \langle \xi \rangle^{s_-}(x,tD) f \Vert_{L^2} + C t \Vert f \Vert_{L^2}.
\end{align*}
By the Calderón-Vaillancourt theorems \cite[Thm. 2]{Hwang87} and \cite[Corol. 2.2]{Hwang87}, we also obtain that
$$
\Vert p_\varepsilon \langle \xi \rangle^{s_-}(x,tD) f \Vert_{L^2} \leq C \Vert f \Vert_{L^2}, \quad t \in (0,1].
$$
Moreover, Corollary \ref{c:now_with_sobolev} remains valid for $p^{-1}_\varepsilon$ instead of $p^{-1}$. By Corollary \ref{c:now_with_sobolev}, we also have that
$$
\Vert \mathfrak{C}_t(p_\varepsilon,p^{-1}_\varepsilon) \Vert_{\mathcal{L}(H^{s_-}_t,H^{s_-}_t)} \leq tC.
$$
Thus, we can apply Proposition \ref{l:coercividad} (notice that we can replace $p^{-1}$ by $p_\varepsilon^{-1}$ and $H_t^{-m_1}$ by $H^{s_-}_t$), to obtain, for every $0 < t \leq T$,
$$
\Vert p^{-1}_\varepsilon(x,tD) (\chi_\varepsilon f) \Vert_{L^2} \geq C_1 \Vert \chi_\varepsilon f \Vert_{H^{s_-}_t(\R)}.
$$
It remains to show the inequality on the right. To this aim, we use again \eqref{e:localize} and \eqref{e:commutator_with_bump}, the facts that
$$
p^{-1}_\varepsilon(x,tD) (\chi_\varepsilon f)(x) = p^{-1}_\varepsilon(x,tD)  \langle tD \rangle^{-s_+} \langle tD \rangle^{s_+} (\chi_\varepsilon f)(x),
$$
and that $p^{-1}_\varepsilon(x,tD)  \langle tD \rangle^{-s_+}$ is the semiclassical quantization of the symbol $p^{-1}_\varepsilon(x,\xi)\langle \xi \rangle^{-s_+}$.  It is then sufficient to use  \cite[Thm. 2]{Hwang87} and \cite[Corol. 2.2]{Hwang87}) to show that $p^{-1}_\varepsilon(x,tD)  \langle tD \rangle^{-s_+}$ is bounded from $L^2(\R)$ to $L^2(\R)$ uniformly for $0< t \leq T$. This finishes the proof.

\end{proof}

%\bibliographystyle{plain}
%\bibliography{Referencias}

\end{document}